\documentclass{sn-jnl}
\usepackage{lmodern}
\usepackage{graphicx}%
\usepackage{multirow}%
\usepackage{amsmath,amssymb,amsfonts}%
\usepackage{amsthm}%
\usepackage{mathrsfs}%
\usepackage[title]{appendix}%
\usepackage{xcolor}%
\usepackage{textcomp}%
\usepackage{manyfoot}%
\usepackage{booktabs}%
\usepackage{algorithm}%
\usepackage{algorithmicx}%
\usepackage{algpseudocode}%
\usepackage{listings}%
\usepackage{mathtools}
\mathtoolsset{showonlyrefs}
\usepackage{comment}

\theoremstyle{plain}
\newtheorem{theorem}{Theorem}[section]
\newtheorem{lemma}[theorem]{Lemma}
\newtheorem{proposition}[theorem]{Proposition}
\newtheorem{corollary}[theorem]{Corollary}

\makeatletter
\newtheorem*{rep@theorem}{\rep@title}
\newcommand{\newreptheorem}[2]{%
\newenvironment{rep#1}[1]{%
 \def\rep@title{#2 \ref{##1}}%
 \begin{rep@theorem}}%
 {\end{rep@theorem}}}
\makeatother
\newreptheorem{theorem}{Theorem}
\newreptheorem{corollary}{Corollary}

\theoremstyle{definition}
\newtheorem{example}{Example}%
\newtheorem{definition}{Definition}%

\theoremstyle{remark}
\newtheorem{remark}{Remark}%

\newcommand{\R}{\mathbb{R}} 
\def \Rn{{\R^n}}
\newcommand{\hid}{m}
\def \Rnhi{{\R^{n\hid}}}

\def \B{B_2^n}
\newcommand{\N}{\mathbb{N}}
\def \E{\mathbb E}

\def \s{\mathbb{S}^{n-1}}
\def \S{\mathbb{S}^{n\hid-1}}
\def \u{\underline}
\def \o{\overline}
\def \conbod {\mathcal{K}^n}
\def \Conbod {\mathcal{K}^{n \hid}}
\newcommand{\vol}{\text{\rm Vol}}

\def \p{\Pi}
\def \pp{\Pi^{\circ}}
\def \g{\Gamma}

\def \cov{g_{K}}

\newcommand {\LYZ}[1]{\P{\langle#1\rangle}}
\newcommand {\LYZP}[1]{\PP \langle#1\rangle}

\def \conv{\operatorname{conv}}

\def \P{\Pi^{\hid}}
\def \PP{\Pi^{\circ,\hid}}
\def \G{\Gamma^{\hid}}

\def \sta{\mathcal{S}^{n\hid}}

\def \Cov{g_{K,\hid}}

\begin{document}

\title[Higher Order Projection and Centroid Bodies]{Affine Isoperimetric Inequalities for Higher-Order Projection and Centroid Bodies}


\author[1]{\fnm{Juli\'an} \sur{Haddad}}\email{jhaddad@us.es}

\author*[2]{\fnm{Dylan} \sur{Langharst}}\email{dylan.langharst@imj-prg.fr}

\author[3]{\fnm{Eli} \sur{Putterman}}\email{putterman@mail.tau.ac.il}

\author[4]{\fnm{Michael} \sur{Roysdon}}\email{mar327@case.edu}

\author[5]{\fnm{Deping} \sur{Ye}}\email{deping.ye@mun.ca}

\affil[1]{\orgdiv{Departamento de An\'alisis Matem\'atico}, \orgname{Universidad de Sevilla}, \orgaddress{\street{C. Tarfia}, \city{Sevilla}, \postcode{41012}, \country{Spain}}}

\affil*[2]{\orgdiv{Institut de Math\'ematiques de Jussieu}, \orgname{Sorbonne Universit\'e}, \orgaddress{\street{4 Place Jussieu}, \city{Paris}, \postcode{75252}, \country{France}}}

\affil[3]{\orgdiv{School of Mathematical Sciences}, \orgname{Tel Aviv University}, \orgaddress{\city{Tel Aviv}, \postcode{66978}, \state{Tel Aviv-Yafo}, \country{Israel}}}

\affil[4]{\orgdiv{Department of Mathematics, Applied Mathematics, and Statistics}, \orgname{Case Western Reserve}, \orgaddress{ \city{Cleveland}, \postcode{44106}, \state{Ohio}, \country{USA}}}

\affil[5]{\orgdiv{Department of Mathematics and Statistics}, \orgname{Memorial University of Newfoundland}, \orgaddress{\city{St. John's}, \state{Newfoundland}, \postcode{A1C 5S7}, \country{Canada}}}

\abstract{In 1970, Schneider introduced the $\hid$th order difference body of a convex body, and also established the $\hid$th-order Rogers-Shephard inequality. In this paper, we extend this idea to the projection body, centroid body, and radial mean bodies, as well as prove the associated inequalities (analogues of Zhang's projection inequality, Petty's projection inequality, the Busemann-Petty centroid inequality and Busemann's random simplex inequality). We also establish a new proof of Schneider's $\hid$th-order Rogers-Shephard inequality. As an application, a $\hid$th-order affine Sobolev inequality for functions of bounded variation is provided.}

\keywords{ Projection bodies, Centroid bodies, Radial mean bodies, Busemann-Petty centroid inequality, Steiner symmetrization, Petty projection inequality, affine Sobolev inequality}


\pacs[MSC Classification]{52A39, 52A40;  Secondary: 28A75, 46E35}

\maketitle

\section{Introduction}
\label{sec:intro}
In this work, we will be dealing with convex bodies. A set $K$ in $\R^n$ (the standard $n$-dimensional Euclidean space) is said to be a \textit{convex body} if it is a compact set with non-empty interior such that, for every $x,y\in K$ and $\lambda\in[0,1],$ $(1-\lambda)y+\lambda x\in K.$ The set of convex bodies in $\R^n$ will be denoted by $\conbod$. A convex body $K$ is origin-symmetric if $K=-K$, and is said to be symmetric if a translate of $K$ is origin-symmetric.
\subsection{The Difference Body and the Rogers-Shephard inequality}
The \textit{difference body} of $K\in\conbod$ is given by
	\begin{equation}\label{e:differnce}
	    DK:=\{x:K\cap(K+x)\neq \emptyset\}=K+(-K),
	\end{equation}
 where $K+L=\{ a + b: a\in K, b\in L\}$ is the Minkowski sum of $K,L\in\conbod$. Throughout, $\vol_n(K)$ denotes the volume (Lebesgue measure) of $K.$ The functional $K \mapsto \vol_n(DK)/\vol_n(K)$ is continuous and affine invariant on $\conbod$, and can be taken as a measure of asymmetry of $K\in\conbod.$ Indeed, one has
	\begin{equation}\label{e:roger_shep}
	2^n\leq\frac{\vol_n(DK)}{\vol_n(K)}\leq\binom{2n}{n},
	\end{equation}
	where the left-hand side follows from the Brunn-Minkowski inequality (see Section~\ref{sec:properties} below), with equality if and only if $K$ is symmetric, and the right-hand side is the \textbf{Rogers-Shephard inequality}, with equality if and only if $K$ is a $n$-dimensional simplex (closed convex hull of $n+1$ affinely independent points) \cite{RS57}. Here, $\Gamma(z)$ is the standard gamma function and $\binom{m}{n}=\frac{\Gamma(m+1)}{\Gamma(n+1)\Gamma(m-n+1)}$ are the binomial coefficients.

 The starting point of this work is the following. For $\hid\in\mathbb{N}$, consider the higher-dimensional difference body, given by
\[
D^\hid(K) := \left\{\overline{x}=(x_1,\dots,x_{\hid}) \in (\R^n)^{\hid} \colon K \cap \bigcap_{i=1}^\hid (x_i+K) \neq \emptyset \right\}.
\]

We will henceforth refer to $D^\hid(K)$ as the $\hid$\textit{th-order difference body of K}. We emphasize that this is a convex body in $(\R^{n})^\hid=\R^{n\hid}$.
Throughout this paper we will simply write $\R^{n\hid},$ but the product structure $\R^{n\hid} = \R^n \times \cdots \times \R^n$ will be of central importance. We will often use $\bar{x}=(x_1,\dots,x_\hid)$, with each $x_i\in\Rn$, to denote a point in $\Rnhi$.
In \cite{Sch70}, Schneider established the following $\hid$th-order Rogers-Shephard inequality for $K\in\conbod$:
\begin{equation}
    \frac{\vol_{n\hid}\left(D^\hid (K)\right)}{\vol_n(K)^{\hid}}\leq \binom{n\hid +n}{n}
    \label{eq:RSell},
\end{equation}
with equality if and only if $K$ is a $n$-dimensional simplex. From the affine invariance and continuity of the functional in the left-hand side of \eqref{eq:RSell}, a minimum for this quantity also exists. When $n=2$ and $\hid\in\mathbb{N}$ is arbitrary, Schneider was able to show that the minimum is attained for every symmetric $K\in\mathcal{K}^2,$ but, for $n\geq 3,$ $\hid\geq 2,$ he constructed an example of a symmetric body where the value of the functional in \eqref{eq:RSell} is strictly larger than its value when $K$ is an ellipsoid. Thus, it is conjectured that \eqref{eq:RSell} is minimized for ellipsoids for such $n$ and $\hid$. 
The objective of this paper is to extend the theory of mth-order convex bodies to other operators defined on $\conbod$ and $\Conbod$, namely the projection body, the centroid body and the radial mean bodies.

\subsection{The projection body and related inequalities}
Let $B_2^n$ denote the unit Euclidean ball in $\R^n$ and denote its boundary as $\mathbb{S}^{n-1},$ for $n\in\N.$ Using the notation $\langle x,y \rangle$ for the Euclidean inner-product of $x,y \in \Rn,$ we have that for $K \in \conbod$, the orthogonal projection of $K$ onto $\theta^{\perp}=\{x\in\R^n:\langle \theta,x\rangle =o\},$ the subspace through the origin orthogonal to $\theta\in\s,$ is denoted by $P_{\theta^\perp} K$. The support function of $L\in\conbod,$ defined as 
\[h_L(x)=\sup \{\langle x,y\rangle \colon y \in L \}, \quad x\in \Rn,\]
uniquely determines $L.$  Recall also the following notation: given a function $f$, we can write $f=f_+-f_{-}$, where $f_+=\max\{f,0\}$ and $f_-=-\min\{f,0\}$. The \textit{projection body} of $K\in\conbod$, denoted $\Pi K$, is the origin-symmetric convex body whose support function is given by any of the following equivalent definitions:
\begin{equation}
 \label{eq:cauchy}
	   h_{\Pi K}(\theta)=\vol_{n-1}\left(P_{\theta^{\perp}}K\right) =\int_{\partial K} \langle \theta, n_K(y) \rangle_{-}dy = nV_n(K[n-1], [o,-\theta]),
	\end{equation}
 for $\theta\in\s$.
Here, $\partial K$ denotes the boundary of $K,$ $dy$ represents integration with respect to the $(n-1)$-dimensional Hausdorff measure on $\partial K$, $n_K(y)$ represents the outward-pointing unit normal vector, which is defined a.e. on $\partial K$, and $V_n(K[n-1],[o,-\theta])$ denotes the mixed volume of $(n-1)$ copies of $K$ and the line segment connecting the origin $o$ and $-\theta,$ denoted $[o,-\theta]$ (see Section~\ref{sec:properties} for the definition of mixed volumes).

An important function associated with a convex body $K$ is its \textit{covariogram function}, which appears often in the literature. We recall its definition: for $K\in\conbod$ the \textit{covariogram} of $K$ is given by
	\begin{equation}\label{e:covario}
	    \cov(x)=\vol_n\left(K\cap(K+x)\right)=(\chi_K\star\chi_{-K})(x)
	\end{equation}
 (recall that for a set $K\subset \Rn,$ $\chi_K(y)=1$ if $y\in K$ and $0$ otherwise and
 $(f\star g)(x)= \int_{\R^n} f(y)g(x-y)dy$ is the convolution of functions $f,g:\R^n \to \R$.)
%	$\chi_K(x)$ is the characteristic function of $K$. 
	%given by
	%\begin{align*}\chi_K(x)=
	%\begin{cases}
%	1 \; \text{if } x\in K
%	\\
%	0 \; \text{if } x\notin K
%	\end{cases}
%	\end{align*}
	%\noindent Being defined as a convolution of characteristic functions, the covariogram inherits the $1/n$-concavity of the Lebesgue measure on its support.
 The support of $\cov(x)$ is the difference body of $K$, $DK$. In fact, Chakerian \cite{Cha67} was the first to show the connection between the covariogram and reverse-isoperimetric-type inequalities in his celebrated, concise proof of the Rogers-Shephard inequality. The covariogram function has the following radial derivatives:
 $$\frac{d\cov(r\theta)}{dr}\bigg|_{r=0^+}=-h_{\Pi K}(\theta), \quad \theta\in\s.$$
This formula was first shown by Matheron \cite{MA}.

For a convex body $L$ having the origin as an interior point, its polar body is $L^\circ=\left\{x\in \R^n: h_L(x) \leq 1\right\}$.  Write $\kappa_n=\vol_n(B_2^n)$ and $(\Pi K)^\circ=\Pi^\circ K$. One has that $\Pi^\circ B_2^n = \kappa_{n-1}^{-1}B_2^n.$ The Petty product $K\mapsto \vol_n(K)^{n-1}\vol_n(\Pi^\circ K)$ is a continuous, affine invariant functional on $\conbod.$ For $K\in\conbod$, one has
\begin{equation}\label{e:Zhang_petty_ineq}
\frac{1}{n^n} \binom{2n}{n}\leq\vol_n(K)^{n-1}\vol_n(\Pi^\circ K)\leq \left(\frac{\kappa_n}{\kappa_{n-1}}\right)^n.
\end{equation}
The right-hand side of \eqref{e:Zhang_petty_ineq} is \textbf{Petty's projection inequality} which was proven by Petty in 1971 \cite{CMP71}; equality occurs if and only if $K$ is an ellipsoid. The left-hand side of \eqref{e:Zhang_petty_ineq} is known as \textbf{Zhang's projection inequality}, proven by Zhang in 1991 \cite{Zhang91}; equality holds if and only if $K$ is a $n$-dimensional simplex.

An easy consequence of the definition of the projection body operator is \textbf{Petty's isoperimetric inequality} \cite{petty61_1,CMP71}, which asserts that, for every $K\in\conbod,$
\begin{equation}\vol_n(\pp K)\vol_{n-1}(\partial K)^n \geq \kappa_n \left(\frac{\kappa_n}{\kappa_{n-1}}\right)^n,
\label{eq:petty_theorem}
\end{equation}
with equality if and only if $\Pi K$ is a dilate of the Euclidean ball. In fact, Petty's isoperimetric inequality \eqref{eq:petty_theorem}, in conjunction with Petty's projection inequality, the second inequality in \eqref{e:Zhang_petty_ineq}, implies the classical isoperimetric inequality.

Prominent extensions of the projection body operator include the $L^p$ projection bodies, where the Petty projection inequality for these bodies was established by Lutwak, Yang and Zhang \cite{LYZ00}, and asymmetric $L^p$ projection bodies, introduced by Lutwak \cite{LE96} and Ludwig (from a valuation-theory perspective) \cite{ML05}; the Petty projection inequality for these latter bodies was established by Haberl and Schuster \cite{HS09}.

\subsection{Higher-order projection body}
We now introduce the extension of the covariogram, first defined implicitly by Schneider \cite{Sch70}. One has, for $K\in\conbod$ and $\hid\in\mathbb{N},$ that the $\hid$-covariogram of $K$ is given by, using the notation $\bar x=(x_1,\dots,x_\hid)\in\Rnhi$:
\begin{equation}
\Cov(\overline{x})=\vol_n\left(K\cap\bigcap_{i=1}^\hid (x_i+K)\right).
\label{eq_vol_ell_co}
\end{equation}
Our first main result gives a formula for the radial derivative of the $\hid$-covariogram.  For $\bar\theta=(\theta_1,\dots,\theta_\hid),$ we introduce the notation $$C_{\bar\theta}=\conv_{1\leq i \leq \hid}[o,\theta_i],$$ where $\conv_{1\leq i \leq \hid}(A_i)$ denotes the closed convex hull of the sets $A_1,\dots,A_\hid.$ Notice that $\max_{1 \leq i \leq \hid} \langle \theta_i, v \rangle_{-} = h_{C_{-\bar\theta}}(v).$

\begin{theorem} \label{t:variationalformula}
Let $K \in \conbod$ and $\hid \in \N$. Then, for every direction $\bar{\theta} = (\theta_1,\dots,\theta_{\hid})$ $\in \S$:
\[
\frac{d}{dr} \Cov (r\bar{\theta}) \bigg|_{r=0^+}
= -\int_{\partial K} \max_{1 \leq i \leq \hid} \langle \theta_i, n_K(y) \rangle_{-}dy.\]
\end{theorem}

We next introduce a new generalization of the projection body.
\begin{definition}
\label{def:higher_proj}
    For $K\in\conbod$ and $\hid\in\N$, the \textit{$\hid$th-order projection body} $\P K$ is an $n\hid$-dimensional convex body whose support function is defined for $\bar{\theta}=(\theta_1,\dots,\theta_{\hid})\in\S$ as
\begin{align*}h_{\P K}(\bar{\theta})= nV_n(K[n-1],C_{-\bar\theta})
=\int_{\partial K} \max_{1 \leq i \leq \hid} \langle \theta_i, n_K(y) \rangle_{-}dy.
\end{align*}
\end{definition}
From the second equality, it is easy to see that $\bar\theta\mapsto h_{\P K}(\bar\theta)$ is sublinear, which guarantees the existence of $\P K$ and that it is a compact, convex set. Next, note that $\Pi^{1} K=\Pi K$. From the fact that $$h_{\P K}(\bar{\theta}) \geq \max_{1\leq i \leq \hid} h_{\p K}(\theta_i)>0,$$ we see that $\P K$ contains the origin as an interior point. In particular, $\P K\in \Conbod$.
%Thus, $o\in\text{int}(\P K).$
The translation invariance of mixed volumes shows that $\P (K+x) = \P K$ for every $x\in\Rn.$ For $u\in\s,$ let $u_j=(o,\dots,o,u,o,\dots,o)\in\S$, where $u$ is in the $j$th coordinate. As an easy observation, we see that
$$h_{\P K}(u_j)=nV_n(K[n-1],[o,-u])=h_{\Pi K}(u).$$
This shows in that the projection of $\P K$ onto any of the $\hid$ copies of $\R^n$ is $\Pi K.$ Additionally, taking $u^m=\frac{1}{\sqrt{m}}(u,\dots,u)$ yields $h_{\P K}(u^m)=nV_n(K[n-1],[o,-\frac{u}{\sqrt{m}}])=\frac{1}{\sqrt{m}}h_{\Pi K}(u)$, and thus $\P K$ is not a product set. 

One can see that the dimension of $C_{-\bar\theta}$ exhausts all dimensions from $1$ to $m$ as $\bar\theta$ varies over $\S$. This shows that we cannot consider generalizations of $\P K$ by changing the number of copies of $K$ and $C_{-\bar\theta}$ in the mixed volumes (say, for example, $K$ appearing $(n-2)$ times and $C_{-\bar\theta}$ appearing twice). Compare this to \cite[Equation 5.68]{Sh1} which says, for a $(n-k)$-dimensional subspace $E$ of $\Rn$ $(k=1,\dots,n-1)$ and any convex body $U$ in $E^\perp$ of $k$-dimensional volume $1$, that $\vol_{n-k}(P_E K)=\binom{n}{k}V_n(K[n-k],U[k])$. In particular, this shows we are not in general studying projections of $K$ onto subspaces of co-dimension larger than $1$.

We will use the natural notation $\PP K = (\P K)^\circ$; in Appendix~\ref{secA1}, we compare this body to $(\pp K)^m$. Then, we verify in Proposition~\ref{p:affine_invar_Zhang} that
$$K\in\conbod\mapsto\vol_n(K)^{n\hid-\hid}\vol_{n\hid}(\PP K)$$
is affine invariant on $\conbod.$ A major focus is to then find extremizers of this affine invariant, which in the case of $\hid=1$ return the classical affine isoperimetric inequalities \eqref{e:Zhang_petty_ineq} of Petty and Zhang.

Gardner and Zhang defined in \cite{GZ98} the {\it radial mean bodies} $R_p K$ for $p \in (-1, \infty)$ as a way to interpolate between the bodies $\pp K$ and $D K$ (see Section~\ref{sec:radial}). We shall define the $\hid$th-order radial mean bodies $R_p^\hid K$ (see Section~\ref{sec:new_rad}) and prove the following chain of set inclusions, generalizing \cite[Theorem 5.5]{GZ98}:
%We first show the following chain of set inclusions, where $R_p^\hid K$ is the $(\hid,p)$-Radial Mean body of $K\in\conbod$ (see Section~\ref{sec:radial} for the classic theory and Section~\ref{sec:new_rad} for the generalization). 
\begin{theorem}
\label{t:set_con}
Let $K\in\conbod$ and $\hid\in\mathbb{N}.$ Then, for $-1< p< q < \infty$, one has
$$D^\hid (K) \subseteq {\binom{q+n}{n}}^{\frac{1}{q}} R^\hid_{q}K \subseteq {\binom{p+n}{n}}^{\frac{1}{p}} R^\hid_{p} (K)\subseteq n\vol_n(K)\PP K.$$
Equality occurs in any set inclusion if and only if $K$ is an $n$-dimensional simplex.
\end{theorem}
As a corollary of Theorem~\ref{t:set_con}, we establish the following. 
\begin{corollary}[The mth-order Zhang's projection inequality]
\label{cor:zhanginequality}
    Let $K\in\conbod$ and $\hid\in\mathbb{N}$. Then,
    \[\vol_n(K)^{n\hid-\hid}\vol_{n\hid}\left( \PP K \right) \geq \frac{1}{n^{n\hid}}\binom{n\hid+n}{n},\]
    with equality if and only if $K$ is a $n$-dimensional simplex.
\end{corollary}
We then have the case of the $\hid$th-order Petty's projection inequality, which we prove using a symmetrization technique.
\begin{theorem}[The mth-order Petty's projection inequality]\label{t:pettyprojectioninequality}
Let $\hid \in \N$ be fixed. Then, for every $K \in\conbod,$ one has 
\[
\vol_n(K)^{n\hid-\hid}\vol_{n\hid}(\PP K) \leq \vol_{n}(\B)^{n\hid-\hid}\vol_{n\hid}(\PP \B),
\]
with equality if and only if $K$ is an ellipsoid.
\end{theorem}
We then establish the following generalization of Petty's isoperimetric inequality, \eqref{eq:petty_theorem}. We denote by $w_n(K)$ the mean width of $K\in\conbod$ (see \eqref{eq:mean_wdith}).
\begin{theorem}
\label{t:petty_classic}
    Let $K\in\conbod$ and $\hid\in\mathbb{N}.$ Then, one has the following inequality:
    \begin{align*}\vol_{n\hid}(\PP K)\vol_{n-1}(\partial K)^{n\hid} &\geq  \vol_{n\hid}(\PP B_2^n)\vol_{n-1}( \s)^{n\hid} 
    \\
    &\geq \kappa_{n\hid}\left(\frac{n\kappa_{n}}{w_{n\hid}(\P\B)}\right)^{n\hid}.
    \end{align*}

Equality in the first inequality holds if and only if $\Pi K$ is an Euclidean ball. If $\hid=1$, there is equality in the second inequality, while for $\hid \geq 2,$ the second inequality is strict.
\end{theorem}
\noindent In fact, the first inequality in Theorem~\ref{t:petty_classic}, in conjunction with Theorem~\ref{t:pettyprojectioninequality}, implies the classical isoperimetric inequality in $\Rn$ regardless of the choice of $\hid\in\N$.

There have been, of course, other generalizations of the Rogers-Shephard inequality (see e.g. \cite{AHNRZ,Ro20} for measure-theoretic extensions, \cite{AGLYN22} for an extension to the lattice point enumerator and \cite{AG19,AlGMJV,AAEFO15,Co06} for functional extensions) and Petty's projection inequality (see e.g. \cite{Zhang91} for all compact Borel sets, \cite{YL21} for sets of finite perimeter, \cite{LYZ10} for the Orlicz setting and \cite{HL25} for the fractional calculus-setting). Zhang's projection inequality has also been recently extended to measures \cite{LRZ22}.

\subsection{Centroid bodies and dual mixed volumes}
\label{sec:cent_mix}
Given a compact set $A\subset \Rn$ with positive volume, its \textit{centroid body} $\g A \in \conbod$ is defined by the support function
\[
    h_{\g A}(\theta) = \vol_n(A)^{-1} \int_A |\langle x, \theta \rangle| dx, \quad \theta\in\s.
\]
Petty's inequality (the right-hand side of \eqref{e:Zhang_petty_ineq}) implies via a duality argument (see \eqref{eq:og_dual}) that the continuous and affine invariant operator $K \mapsto \vol_n(\g K)\vol_n(K)^{-1}$ on $\conbod$ is minimized by centered ellipsoids. This lower bound is known as the \textbf{Busemann-Petty centroid inequality}.

The Busemann-Petty centroid inequality can be interpreted as a bound on a random process.
Let $X_i \in \Rn, i=1, \ldots n,$ be independent random vectors uniformly distributed inside $K$.
%Let $[o,X_1,\dots X_n]$ denote the simplex whose vertices are $o,X_1,\dots X_n.$
We denote the expected volume of $C_{\bar X}=\conv_{1\leq i \leq \hid}[o,X_i]$, a \textit{random simplex of $K$}, by
\begin{equation}
    \E_{K^n}(\vol_n(C_{\bar X})):= \vol_n(K)^{-n} \int_K\cdots \int_K \vol_n\left(\conv_{1\leq i \leq n}[o,x_i]\right) dx_1\dots dx_n.
    \label{eq:random}
\end{equation}
Petty \cite{petty61_2} showed the right-hand side equals $2^{-n} \vol_n(\Gamma K)$.
Thus, the Busemann-Petty centroid inequality is equivalent to the \textbf{Busemann random simplex inequality} \cite{Busemann53}:
\begin{equation}
    \E_{K^n}(\vol_n(C_{\bar X}))\vol_n(K)^{-1}\geq \left(\frac{\kappa_{n-1}}{(n+1)\kappa_n}\right)^n,
\end{equation}
with equality if and only if $K$ is a centered ellipsoid. 
    Our last main result is the mth-order extension of the centroid body. We will define the \textit{$\hid$th-order centroid body of $L\in\Conbod$,} denoted $\G  L,$ via a duality relation with $\PP K$ for $K\in\conbod$ (see Section~\ref{sec:petty}). 
\begin{definition}
\label{def:mth_centroid}
        For a compact set $L \subset \Rnhi$ with positive volume, its $\hid$-centroid body $\G  L \in \conbod$ is given by the support function
        \begin{equation}
        h_{\G  L}(\theta) = \frac{1}{\vol_{n\hid}(L)}\int_L \max_{1\leq i \leq \hid} \langle x_i, \theta \rangle_{-} d\bar{x},
        \label{eq:cen_hi}
        \end{equation}
        where $\bar x = (x_1, \ldots, x_\hid)$ and $\theta\in\s$.        
\end{definition}
Like in the case of the $m$th-order projection body, the fact that $\theta\mapsto h_{\G L}(\theta)$ is sublinear and strictly positive yields that $\G L$ exists and is a convex body. We emphasize that the operator $\G$ sends a $n\hid$-dimensional compact set to a $n$-dimensional convex body, which is dual to the behavior of the operator $\PP.$ We obtain the $\hid$th-order Busemann-Petty centroid inequality, where a new convex body plays the role of minimizer. We will actually prove the result for a class of compact sets larger than the class of convex bodies having the origin in their interior, the so-called star bodies (see Section~\ref{sec:properties} for a precise definition). We will use $\sta$ to denote the class of star bodies in $\Rnhi.$
\begin{theorem}[The mth-order Busemann-Petty centroid inequality] 
\label{t:BPCH}
For $L\in\sta$, where $n,\hid\in\mathbb{N},$ one has
    %$$\vol_n(\G  L)^{\hid} \geq \left(\frac{\hid}{(n\hid +1)\kappa_n}\right)^{n\hid}\frac{\kappa_n^\hid }{\vol_{n\hid}(\PP \B)}\vol_{n\hid}(L),$$
    \[\frac{\vol_n(\G  L)}{\vol_{n\hid}(L)^{1/\hid}} \geq \frac{\vol_n(\G \PP\B)}{\vol_{n\hid}(\PP\B)^{1/\hid}}, \]
    with equality if and only if $L = \PP E$ for any ellipsoid $E\in\conbod$.
\end{theorem}
\noindent Other prominent extensions of the Busemann-Petty centroid inequality have been established in the setting of the Firey-Brunn-Minkowski theory \cite{LYZ00}, functions \cite{HJS21} and the Orlicz-Brunn-Minkowski theory \cite{LYZ10_2}.

We note that $\Gamma^1 L$ is not the usual centroid body $\Gamma L$, but instead the asymmetric $L^1$ centroid body studied by Haberl and Schuster \cite{HS09}. The reason we recover the asymmetric case and not the usual case is because we defined our operators in terms of $[o,-\theta]$ instead of $[-\frac{\theta}{2},\frac{\theta}{2}]$. While this preserves the mixed volumes defining $\p K$, it will not preserve the dual mixed volumes. However, the $m=1$ case of Theorem~\ref{t:BPCH} (which is the $L^1$ case of Haberl and Schuster's asymmetric Busemann-Petty centroid inequality) is sharper than the usual Busemann-Petty centroid inequality. Indeed, notice that $\Gamma^1 (-L)=-\Gamma^1L$, and, therefore, $\Gamma L=D(\Gamma^1 L)$. Then, combining Theorem~\ref{t:BPCH} with the left-hand side \eqref{e:roger_shep}, which is merely the Brunn-Minkowski inequality, yields the usual Busemann-Petty centroid inequality.

After Busemann, Groemer generalized the random simplex inequality. First, he proved in \cite{Groemer73} that the expected value, as well as the higher order moments, of the volume of the closed convex hull of $(n+1)$ points inside $K\in\conbod$ is minimized when $K$ is an ellipsoid. Then, Groemer \cite{Groemer74} extended this result to the case where the number of points is allowed to be greater than or equal to $(n+1)$. Hartzoulaki and Paouris \cite{HP03} proved an analogous result for the expected value of the Querma{\ss}integrals.

There have been many other recent extensions of Groemer's results, see for example \cite{APPS24,APS24,CEFPP15,PP12,PP13-1,PP13-2,PP17-1,PP17-2,PPZ14}. A common feature of all Groemer-type results is that the points $X_i$ are always taken independently, so that the vector $(X_1, \ldots, X_\hid) \in \Rnhi$ is distributed with respect to a product probability measure on $\Rnhi=\Rn \times \cdots \times \Rn$. Theorem \ref{t:BPCH} lets us obtain a new Groemer-type result, generalizing \cite{HP03} for the mean width. Fix $K \in \conbod, L \in \sta$ and let $\bar X = (X_1, \ldots, X_\hid) \in \Rnhi$ be a random vector uniformly distributed inside $L$, (no independence of the $X_i$ is required).
We denote the expected mixed volume of $K$ and $C_{\bar X}$ by
\[\E_L (V_n(K[n-1], C_{\bar X}):=\frac{1}{\vol_{n\hid}(L)}\int_L V_n(K[n-1],C_{\bar x}) d\bar x.\]

\begin{theorem}
	\label{t:randomsimplex1}
	The functional
	\[(K,L) \in \conbod \times \sta \mapsto \vol_{n\hid}(L)^{-\frac{1}{n\hid}}\vol_n(K)^{- \frac{n-1}n} \E_L (V_n(K[n-1], C_{\bar X}) )\]
	is uniquely minimized when $K$ is an ellipsoid and $L = \lambda \PP K$ for some $\lambda >0$.
\end{theorem}
\noindent We would like to emphasize that Theorem \ref{t:randomsimplex1} is not a generalization nor a particular case of the results in \cite{HP03,PP12, PP17-2}, since we are minimizing in a different class of probability measures.

The convex body $\PP \B \subseteq \Rnhi$ has some interesting properties that we investigate in Section \ref{sec:prop_proj}. Recalling that $w_n$ denotes the mean width, the support function of $\PP\B$ is
\[h_{\P\B}(\bar x) = n\kappa_n w_n(C_{\bar x}).\]
In fact, a special case of Theorem~\ref{t:randomsimplex1} is that the functional 
\begin{equation}
    \label{eq:randomPB}
    \vol_{n\hid}(L)^{-\frac{1}{n\hid}}\E_L(w_n(C_{\bar X}))=\vol_n(L)^{- \frac {n\hid+1}{n\hid}} \int_L w_n( C_{\bar x}) d \bar x
\end{equation}
is minimized for $L = \PP\B$ over $\sta$.
This fact can be regarded as a result of minimization of a random process in the spirit of the Busemann random simplex inequality.
Actually, \eqref{eq:randomPB} represents the expected value of the mean width of $C_{\bar x}$, assuming that $\bar x$ is distributed uniformly on $L$.
If $L = K \times \cdots \times K$ we obtain exactly $\E_{K^\hid}(w_n(C_{\bar X}))$ as defined previously.

\subsection{Affine Sobolev inequalities}

Recall that the $L^p$ norm of an integrable function on $\Rn$ is given by
$$\|f\|_p=\left(\int_{\Rn}|f(x)|^pdx\right)^{1/p}.$$
We follow the standard notation of using $\nabla f$ to denote the gradient of a differentiable function. Recall also that $f$ belongs to the Sobolev space $W^{1,1}(\Rn)$ if it has a weak derivative, i.e. there exists measurable vector map $\nabla f:\Rn\to\Rn$ such that $|\nabla f|$ is integrable and satisfies 
$$\int_{\Rn} f(x) \text{ div} \psi(x) dx = - \int_{\Rn} \langle \nabla f(x), \psi(x) \rangle dx$$
for every compactly supported, smooth vector field $\psi:\Rn \to \Rn$ \cite{Evans}. 

Given a function $f \in W^{1,1}(\Rn)$ which is not zero almost everywhere, there exists a unique convex body $\langle f \rangle$ in $\Rn$ with center of mass at the origin, the LYZ body of $f$, that satisfies the following change of variables formula 
\begin{equation}\int_{\partial \langle f\rangle}g(n_K(y))dy=\int_{\Rn}g(-\nabla f(x))dx
\label{eq:LYZw}
\end{equation}
for every $1$-homogeneous function $g$ (see \cite{LYZ06}) on $\Rn\setminus\{o\}$. We remark that the epithet ``LYZ Body of $f\in W^{1,1}(\Rn)$", first introduced by Ludwig \cite{LUD}, is reserved for the origin-symmetric convex body given by the Blaschke body of $\langle f\rangle$ (see \eqref{eq:BL_body}); in which case, \eqref{eq:LYZw} would only hold for even, 1-homogeneous functions $g$. Lutwak, Yang and Zhang \cite{LYZ06} first defined both sets of bodies (though, the body we denote by $\langle f\rangle$ is actually the reflection of the body implied to exist by \cite[Lemma 4.1]{LYZ06}). The more general, not necessarily symmetric body, is important for our setting. Indeed, by setting $g=h_{C_{-\bar\theta}}$ in \eqref{eq:LYZw}, which is in general not an even function, we obtain the \textit{LYZ projection body of order $\hid$ of $f\in W^{1,1}(\Rn)$}, $\Pi^\hid \langle f\rangle$, defined via
\begin{equation}
h_{\LYZ{f}}(\bar \theta) = \int_{\R^n} \max_{1\leq i \leq \hid}\langle \nabla f(x), \theta_i \rangle_-dx.
\end{equation}

As an application of our results, we show the following affine Sobolev inequality. It generalizes the result of Zhang \cite{GZ99} in the $\hid=1$ case. The result we present here is actually a corollary of a more general result for functions of bounded variation (see Theorem~\ref{t:GeneralAffineSobolev}), which was first done by Wang \cite{TW12} in the $\hid=1$ case.
\begin{corollary}\label{cor:GeneralAffineSobolev} Fix $\hid,n \in \N$. Consider a compactly supported $f\in W^{1,1}(\Rn)$ which is not zero almost everywhere. Then, one has the geometric inequality
\[
\|f\|_{\frac{n}{n-1}} \vol_{n\hid}(\LYZP{f})^\frac{1}{n\hid} \leq \vol_{n\hid}(\PP{\B})^\frac{1}{n\hid}\kappa_n^{\frac{n-1}{n}}.
\]
Equivalently, by setting $
d_{n,\hid} := n\kappa_n\left(nm\vol_{n\hid}(\PP{\B})\right)^\frac{1}{n\hid}$, this is
\begin{equation*}
\left(\int_{\S} \left(\int_{\R^n} \max_{1\leq i \leq \hid}\langle \nabla f(z), \theta_i \rangle_- dz \right)^{-n\hid} d\bar\theta\right)^{-\frac{1}{n\hid}} d_{n,\hid} \geq  n\kappa_n^{\frac{1}{n}}\|f\|_{\frac{n}{n-1}}.
\end{equation*}
\end{corollary}

\noindent An extension of the affine Sobolev inequality was also done to the setting of Firey-Brunn-Minkowski theory \cite{LYZ00} and the fractional calculus \cite{HL25}.

The paper is organized as follows. In Section~\ref{sec:pre}, we discuss more facts about convex bodies and their volume (Section~\ref{sec:properties}) and the classical theory of radial mean bodies (Section~\ref{sec:radial}). In Section~\ref{sec:diff_pro}, we recall the basic properties of the $\hid$-covariogram and prove Theorem~\ref{t:variationalformula} (Section~\ref{sec:co_proj}). In Section~\ref{sec:prop_proj}, we analyze the $\hid$th-order projection body, and then we analyze the $\hid$th-order centroid body in Section~\ref{sec:cent_prop}. Then, in Section~\ref{sec:new_rad}, we present the generalization of radial mean bodies to our setting and prove Theorem~\ref{t:set_con} and Corollary~\ref{cor:zhanginequality}. In Section~\ref{sec:petty}, we use a symmetrization technique to establish Petty's projection inequality, Theorem~\ref{t:pettyprojectioninequality}, and Petty's isoperimetric inequality, Theorem~\ref{t:petty_classic}, in our setting. The equality case is characterized through a class-reduction technique (see Lemma~\ref{l:classreduction} and Lemma~\ref{l:multidimSteiner}). Finally, in Section~\ref{sec:cent}, we prove the Busemann-Petty centroid inequality, Theorem~\ref{t:BPCH}, and the random-simplex inequality, Theorem~\ref{t:randomsimplex1}, in our setting. The proof of the Busemann-Petty centroid inequality relies on a duality (see Lemma~\ref{lem_duality}) between the centroid and the projection operators. Finally, in Section~\ref{sec:Sobolev Inequality}, we exploit Theorem~\ref{t:BPCH} to prove a $\hid$th-order version of the affine Sobolev inequality of Zhang \cite{GZ99} and Wang \cite{TW12} in Theorem~\ref{t:GeneralAffineSobolev}. As a consequence, we are able to extend Theorem~\ref{t:pettyprojectioninequality} to sets of finite perimeter. 

\section{Preliminaries}
\label{sec:pre}
\subsection{Background from Convex Geometry}
\label{sec:properties}
An important fact about volume, the $n$-dimensional Lebesgue measure, is the famed Brunn-Minkowski inequality: as a functional on $\conbod$ equipped with Minkowski summation $+,$ volume is $1/n$-concave, i.e. for $\lambda \in (0,1)$ and convex bodies $K,L\subset\Rn,$ one has
\begin{equation}
\vol_n\left((1-\lambda) K + \lambda L\right)^{1/n} \geq (1-\lambda)\vol_n\left(K\right)^{1/n} + \lambda\vol_n\left(L\right)^{1/n},
\label{eq:BM}
\end{equation}
with equality if and only if $K$ and $L$ are homothetic, i.e., $K=aL+b$ for some $a\in\R$ and $b\in\R^n.$  Minkowski summation is compatible with support functions: one has, for $\alpha,\beta>0$ and $K,L\in\conbod,$ that $h_{\alpha K+\beta L}(x)=\alpha h_K(x)+\beta h_L(x).$

We say that a set $L\subset\R^n$ is a star body if it is a compact set with $o\in \text{int}(L)$, if $[o,x]\subset L$ for all $x\in L$, and if its radial function is continuous when restricted to $\s$. Recall that the radial function of a compact set $L$ is given by $\rho_L(y)=\sup\{\lambda>0:\lambda y\in L\}.$ The radial function satisfies $\rho_{tL}(y)=t\rho_L(y)$ for $t>0$. Similar to how support functions uniquely determine their associated convex bodies, radial functions uniquely determine their star bodies. Every $K\in\conbod$ containing the origin in its interior is a star body. A use of the radial function is that, if $f$ is integrable over a star body $L$, then we have the polar coordinate formula:
\begin{equation}
\int_L f(x)dx = \int_{\s} \int_0^{\rho_L(\theta)}f(r\theta)r^{n-1}drd\theta,
\label{eq:polar}
\end{equation}
where $d\theta$ denotes the spherical Lebesgue measure on $\s$. We isolate the case when $f(x)=1$:
\begin{equation}
\vol_n(L) = \frac{1}{n}\int_{\s} \rho_L(\theta)^nd\theta.
\label{eq:polar_2}
\end{equation}

Similarly, for those $K\in\conbod$ containing the origin in their interior, the \textit{Minkowski functional}, or gauge, of $K$ is defined to be $\|y\|_K=\inf\{r>0:y\in rK\}$, and we have $\|y\|_K=\rho_K(y)^{-1}$ for $y\neq o$.
Notice that $\|x\|_{K^\circ} =h_K(x)$. A classification of norms on $\R^n$ is that a $1$-homogeneous, convex function is a norm if and only if it is the Minkowski functional of an origin-symmetric convex body. The standard Euclidean norm is precisely the Minkowski functional of $\B,$ and in this instance we write $|x|:=\|x\|_{\B}.$ Other classical notions we use include $\mathcal{O}(n),$ the orthogonal group on $\s,$ and $GL_n(\R),$ the group of non-singular $n\times n$ matrices over $\R.$

The following facts about convex bodies can be found in the textbook by Schneider \cite{Sh1}. Steiner's formula states that the Minkowski sum of two compact, convex subsets of $\Rn$ can be expanded as a polynomial of degree $n$: for every $t\in(0,1)$ and compact, convex $K,L\subset\Rn,$ one has
$$\vol_n(K+tL)=\sum_{j=0}^n \binom{n}{j}t^j V_n(K[n-j],L[j]),$$
where $V_n(K[n-j],L[j])$ is the \textit{mixed volume} of $(n-j)$ copies of $K$ and $j$ copies of $L$. When $j=1$, one often writes $V_n(K[n-1],L)$.  Also, when $j=0$, $V_n(K[n])=\vol_n(K).$ The Brunn-Minkowski inequality \eqref{eq:BM} then implies \textit{Minkowski's first inequality}
\begin{equation}V_n(K[n-1],L)^n \geq \vol_n(K)^{n-1}\vol_n(L),
\label{eq:min_first}
\end{equation}
with equality if and only if $K$ and $L$ are homothetic.

A convex set $K$ is often studied through its surface area measure $\sigma_K$: for every Borel $A \subset \s,$ one has $$\sigma_K(A)=\mathcal{H}^{n-1}(n^{-1}_K(A)),$$ where $\mathcal{H}^{n-1}$ is the $(n-1)$-dimensional Hausdorff measure and $n_K:\partial K \rightarrow \s$ is the Gauss map, which associates an element $y$ of $\partial K,$ the boundary of $K,$ with its outer unit normal. For almost all $x\in\partial K$, $n_K(x)$ is well-defined (i.e. $x$ has a single outer unit normal). One can verify via approximation by polytopes that for $K,L\in\conbod,$ 
\begin{equation}
\begin{split}
	    V_n(K[n-1],L):=\frac{1}{n}\lim_{\varepsilon\to0}\frac{\vol_n(K+\varepsilon L)-\vol_n(K)}{\varepsilon}
     =\frac{1}{n}\int_{\s}h_L(\theta)d\sigma_K(\theta).
\end{split}
	    \label{eq:mixed_0}
\end{equation}

\noindent A relevant example of mixed volumes in our case is the \textit{mean width} of $L\in\conbod$, given by
\begin{align}
\label{eq:mean_wdith}
w_n(L)=\frac{1}{n\kappa_n}\int_{\s}h_L(\theta)d\theta=\frac{1}{\kappa_n}V_n(\B[n-1],L).
\end{align}

\noindent Given a convex body $K$, its \textit{Blaschke body} $\widetilde{K}$ is the unique origin-symmetric convex body whose surface area measure is given by \cite[Definition 3.3.8]{gardner_book}
\begin{equation}
\label{eq:BL_body}
d\sigma_{\widetilde{K}}(\theta)=\frac{d\sigma_K(\theta)+d\sigma_K(-\theta)}{2}.
\end{equation}
An application of Minkowski's first inequality is that \cite[Theorem 3.3.9]{gardner_book}
\begin{equation}
    \vol_n(K)\leq \vol_n(\widetilde K),
    \label{eq:BL_volume}
\end{equation}
with equality if and only if $K$ is symmetric, in which case $K$ is a translate of $\widetilde{K}$.

In \cite{Lut75}, Lutwak introduced the $i$th \textit{dual mixed volume} for $K$ and $L$ star bodies in $\Rn$:
\begin{equation}
    \widetilde{V}_i(K[n-i],L[i])=\frac{1}{n}\int_{\s}\rho_{K}(\theta)^{n-i}\rho_L(\theta)^i d\theta.
    \label{eq:dual_mixed}
\end{equation}
Lutwak originally only considered the case when $i\in\mathbb{N},$ $i \leq n,$ but it is not too difficult to expand his theory to the case when $i\in\R,$ see for example the appendix of \cite{gardner_book}. When $i=-1,$ we write $\widetilde{V}_{-1}(K[n+1],L).$ Dual mixed volumes satisfy the \textit{dual Minkowski's first inequality}: for $K,L \in \mathcal{S}^n$,
\begin{equation}
    \vol_n(K)^{n+1}\vol_n(L)^{-1}\leq \widetilde{V}_{-1}(K[n+1],L)^{n},
    \label{dual_Min_first}
\end{equation}
with equality if and only if $K$ and $L$ are dilates (see e.g. \cite[Section B.4, pg. 421, eq. B.24]{gardner_book}). Both the inequality \eqref{dual_Min_first} and its equality characterization follow from H\"older's inequality.

\noindent We refer the reader to \cite{gardner_book,Gr,AK05,KY08,Sh1} for more classical definitions and properties of convex bodies and corresponding functionals. 

 To emphasize how the dual mixed volumes will be used in this work, it is known \cite{gardner_book} that the centroid operator $\g$ and the polar projection operator $\pp$ satisfy the following duality
 \begin{equation}
 \label{eq:og_dual}
 \widetilde V_{-1}(L[n+1], \pp K) = \vol_{n}(L) \frac{n+1}{2} V_n(K[n-1], \g  L),\end{equation}
 for every convex body $K$ and star body $L$ in $\Rn$.

Let us conclude this subsection with a historical remark on one possible proof of Petty's projection inequality, the second inequality in \eqref{e:Zhang_petty_ineq}. Given $u\in\s$, the \textit{Steiner symmetrization} of $K$ about $u^\perp$ constructs the \textit{Steiner symmetral} of $K$, denoted $S_u K$. The set $S_u K$ is a convex body that is symmetric about $u^\perp$ and satisfies $\vol_n(S_u K)=\vol_n(K).$ A well-known theorem \cite[Theorem 6.6.5]{Web94} states that there exists a sequence of directions $\{u_j\}_{j=1}^\infty\subset \s$ such that, if we define $S_1 K = S_{u_1} K, S_{j}K=S_{u_j} S_{j-1}K,$ then $S_j K \to \kappa_n^{-1/n}\vol_n(K)^{1/n} \B$ in the Hausdorff metric (see \cite{Sh1}), i.e., $K$ is transformed into the centered Euclidean ball of the same volume. A fact that seems to have been folklore for a long time, but, as far as the authors are aware, was only first shown rigorously in \cite{LYZ00}, is that for $u\in \s$ and $K\in\conbod,$ one has $$S_u \Pi^\circ K \subseteq \Pi^\circ S_u K, $$
and therefore Petty's inequality, the right-hand side of \eqref{e:Zhang_petty_ineq}, is an immediate consequence of Steiner symmetrization. Inspired by this classical proof, we will recall another type of symmetrization to settle Theorem~\ref{t:pettyprojectioninequality}.

\subsection{Radial Mean Bodies}
\label{sec:radial}
Next, we recall that a non-negative function $\psi$ is said to be $s$-concave, $s>0,$ if for every $x,y\in\text{supp}(\psi)$ and $\lambda\in[0,1],$ one has $$\psi((1-\lambda)x+\lambda y)^s \geq (1-\lambda)\psi(x)^s +\lambda \psi(y)^s.$$
Furthermore, in the limit as $s\to 0^+,$ one has log-concavity:
$$\psi((1-\lambda)x+\lambda y) \geq \psi(x)^{1-\lambda}\psi(y)^\lambda.$$
Jensen's inequality shows that every $s$-concave measure, $s>0,$ is also log-concave. There is an interesting connection between log-concave functions and convex bodies in $\R^n.$ 
\begin{proposition}[Theorem 5 in \cite{Ball88} and Corollary 4.2 in \cite{GZ98}]
\label{p:radial_ball}
    Let $f$ be a log-concave function on $\Rn$. Then, for every $p> 0,$ the function on $\s$ given by 
    $$\theta\mapsto\left(p\int_0^\infty f(r\theta)r^{p-1}dr\right)^{1/p}$$
    defines the radial function of a convex body containing the origin in its interior.
\end{proposition}
\noindent As an interesting historical development, we remark that in \cite{CEFPP15}, it was shown that if the concavity of the function $f$ is related to the power $p,$ then the integral in Proposition~\ref{p:radial_ball} is the radial function of a convex body when $f$ is $-1/(p+1)$-concave.

Gardner and Zhang \cite{GZ98} defined the \textit{pth radial mean body} of a convex body $K$ as the star bodies whose radial function is given by, for $\theta\in\s,$ 
\begin{equation}
    \rho_{R_p K}(\theta)=\left(\frac{1}{\vol_n(K)}\int_K \rho_{K-x}(\theta)^p dx\right)^\frac{1}{p}.
    \label{pth}
\end{equation}
A priori, the above is valid for $p>-1, p\neq 0$.
But also, by appealing to continuity, Gardner and Zhang were able to define $R_0 K$ and establish that
$R_\infty K=DK$. Additionally, Gardner and Zhang established that each $R_p K$ is an origin-symmetric convex body for $p \geq 0$ by applying Proposition~\ref{p:radial_ball} to the covariogram, which is even (the convexity for $p\in (-1,0)$ is still open). By using Jensen's inequality, they then obtained that for $-1 < p < q\leq \infty,$ one has
\begin{equation}
R_{p} K \subseteq R_{q} K \subseteq R_{\infty} K=D K.
\label{eq:radial_set_inc}
\end{equation}
One sees that $R_p K$ tends to $\{o\}$ as $p\to -1.$ Therefore,  Gardner and Zhang defined another family of star bodies depending on $K\in\conbod$, the \textit{pth spectral mean bodies} of $K.$ It turned out the $0$th spectral mean body is $e R_0 K$ and the $p$th spectral mean body for $p\in (-1,0)\cup (0,\infty)$ is $(p+1)^\frac{1}{p}R_p K.$ With this renormalization, one obtains $(p+1)^\frac{1}{p}R_p K \to \vol_n(K) \Pi^\circ K$ as $p\to (-1)^+$, showing that the \textit{shape} of $R_p K$ indeed tends to that of a multiple of the polar projection body of $K$ as $p\to (-1)^+$. Gardner and Zhang then obtained a reverse of \eqref{eq:radial_set_inc}. They accomplished this by using a reverse H\"older-type inequality, Berwald's inequality \cite{Berlem,Bor73} obtaining \cite[Theorem 5.5]{GZ98}, for $-1 < p < q < \infty,$ that
\begin{equation}
    \label{e:set_inclusion}
    D K \subseteq \binom{n+q}{q}^\frac{1}{q} R_{q} K \subseteq \binom{n+p}{p}^\frac{1}{p} R_{p} K \subseteq n \vol_n(K) \Pi^{\circ} K,
\end{equation}
There is equality in each inclusion in \eqref{e:set_inclusion} if and only if $K$ is a $n$-dimensional simplex.

\noindent Radial mean bodies themselves were recently generalized to the setting of Borel measures with density and concavity \cite{LP25} and the fractional calculus \cite{HL25}.

\subsection{Wulff Shapes} 
\label{sec:wulff}
Let $C(\s)$ denote the set of continuous functions from $\s$ to $\R$. For $f\in C(\s)$, the Wulff shape of $f$ is 
$$[f] = \bigcap_{u \in \s} \{x \in \mathbb R^n: \langle x, u\rangle \le f(u)\}.$$
 In most works, $f$ is assumed to be strictly positive, to ensure that $[f]$ has non-empty interior (in particular, $o\in \text{int}( [f])$); see, for example, the textbook by Schneider \cite[Section 7.5, pg. 410]{Sh1}. A notable case is when $f=h_K+th$, where $K$ is a convex body containing the origin in its interior, $h\in C(\s)$, and $t$ is small enough so that $f$ is strictly positive. Actually, one can drop the assumption that $K$ contains the origin and $[h_K+th]$ will still have non-empty interior. Indeed, let $z\in \R^n$ be so that $K-z$ contains the origin in its interior - for example, $z$ being the barycenter of $K$. Then, $h_{K-z}(u)=h_K(z)-\langle z,u\rangle$. Therefore, by picking $t>0$ small enough so that $h_{K-z}+th>0$ on $\s$, we have
\begin{equation}
\label{eq:Wulf_shift}
[h_K+th]=[h_{K-z}+th+\langle z,\cdot\rangle]=[h_{K-z}+th]+z,
\end{equation}
where, in the last line, we used the easily verifiable fact that, if $f$ is so that $[f]$ has non-empty interior, then $[f+\langle z,\cdot\rangle]=[f]+z$.

We next recall the following classical fact about the variation of the volume of Wulff shapes. Usually, it is stated with $o\in \text{int}(K)$. However, \eqref{eq:Wulf_shift} combined with the translation invariance of the Lebesgue measure allows us to drop this assumption.
\begin{lemma}[Aleksandrov's Variational Formula, \cite{AL}]
	\label{l:second}
	For $f\in C(\s)$ and $K\in\conbod$, one has
	$$\lim_{t\rightarrow 0}\frac{\vol_n([h_K+tf])-\vol_n(K)}{t}=\int_{\s}f(u)d\sigma_K(u).$$
	\end{lemma}	

\section{The mth-Order Projection and Centroid Bodies}
\label{sec:diff_pro}
\subsection{The Covariogram}
\label{sec:co_proj}
We first recall the definition of the generalized covariogram, which was first implicitly defined by Schneider \cite{Sch70}. 

\begin{definition} Let $K$ be a convex body in $\R^n$ and $\hid \in \N$. We define the $\hid$-covariogram of the body $K$ to be the function $\Cov \colon (\R^n)^{\hid} \to \R^+$ given by 
\[
\Cov (\bar{x}) = \Cov(x_1,\dots,x_{\hid}) = \int_{K} \left(\prod_{i=1}^{\hid} \chi_{K}(y-x_i)\right)dy.
\]
\label{eq_higher_co}
\end{definition}
In the following proposition, we recall rudimentary facts about $\Cov$ and $D^\hid (K).$ They are easily verifiable by the reader, and can also be found in \cite{Sch70}.
\begin{proposition}
\label{p:facts}
    Fix $\hid \in \N$ and consider $K\in \conbod$. Then: 
    \begin{enumerate}
        \item[1.)]  $\Cov$ is supported on $D^m(K)$ and, for every $z\in\R^n$, $g_{K+z,m}=\Cov$.
        \item[2.)] $D^m(K) \in \mathcal{K}^{nm}$ and, for every $z\in\R^n$, $D^m(K+z)=D^m(K)$.
        \item[3.)] For $m \geq 2, D^m(K)$ is origin-symmetric if and only if $K$ is symmetric.
        \item[4.)] $\Cov$ is $1 / n$-concave on its support.
    \end{enumerate}
\end{proposition}
% \begin{proof}
% It suffices to verify the following set inclusion: for any $t \in [0,1]$ and any pair $\overline{x}= (x_1,\dots,x_{\hid}),\overline{y} = (y_1,\dots,y_{\hid}) \in D^\hid(K)$, it must be the case that \[
% K^t(\overline{x}, \overline{y}) \supset (1-t)  \left[K \cap \left(\bigcap_{i=1}^{\hid} (x_i +K)\right) \right] + t \left[K \cap \left(\bigcap_{i=1}^{\hid} (y_i +K)\right) \right],
% \]
% where 
% \[
% K^t(\overline{x}, \overline{y}) = K \cap \left[ \bigcap_{i=1}^{\hid} ((1-t)x_i+ty_i + K) \right].
% \]
% Indeed, once this set inclusion is established, the inequality follows from \eqref{eq:BM} applied in dimension $n\hid.$

% Let $\Bar{z} \in (1-t)  \left[K \cap \left(\bigcap_{i=1}^{\hid} (x_i +K)\right) \right] + t \left[K \cap \left(\bigcap_{i=1}^{\hid} (y_i +K)\right) \right].$ Then $\Bar{z} = (1-t) z+tz'$, with 
% \[
% z \in K \cap \left(\bigcap_{i=1}^{\hid} (x_i +K)\right) \text{ and } z' \in K \cap \left(\bigcap_{i=1}^{\hid} (y_i +K)\right).
% \]
% By the convexity of $K$, we see that $\Bar{z} \in K$. For each $i =1,\dots,\hid$ there exist $l_i,l_i' \in K$ such that $z=x_i+l_i$ and $z'=y_i+l_i'$, which means that $\Bar{z} = (1-t)x_i +t y_i + ((1-t)l_i+tl_i')$ holds for every $i = 1,\dots,\hid$. Again, using the convexity of $K$, it follows that $\Bar{z} \in K^t(\bar{x},\Bar{y}),$ as required. 
% \end{proof}

We now compute the radial derivative of the $\hid$-covariogram, that is prove Theorem~\ref{t:variationalformula}. For the convenience of the reader, we restate it here.
\begin{reptheorem} {t:variationalformula}
Let $K$ be a convex body in $\R^n$ and $\hid \in \N$. Then, for every fixed direction $\bar{\theta} = (\theta_1,\dots,\theta_{\hid}) \in \S$, one has 
\[
\frac{d}{dr} \Cov (r\bar{\theta}) \bigg|_{r=0^+}
= -\int_{\partial K} \max_{1 \leq i \leq \hid} \langle \theta_i, n_K(y) \rangle_{-}dy.\]
\end{reptheorem}

The main tool we use in the proof of Theorem \ref{t:variationalformula} is the formula for the first variation of the volume of a Wulff shape, see Section~\ref{sec:wulff}. For $K\in\conbod,$ $r\geq 0$ and $\bar\theta\in\mathbb{S}^{nm-1}$, we define the convex set in $\R^n$
$$K_r(\bar\theta)=\left[h_K-r\max_{0\le i \le \hid}\langle \theta_i,\cdot \rangle_- \right].$$
For $r\geq 0$ small enough, $K_r(\bar\theta)$ is a convex body. We can now prove our first theorem. 
\begin{proof}[Proof of Theorem~\ref{t:variationalformula}]
It suffices to show, for every $\bar\theta\in\S$ and $r\in [0,\rho_{D^m(K)}(\bar\theta)]$,
\begin{equation}
\Cov(r\bar\theta)=\vol_n(K_r(\bar\theta)).
\label{eq:covariogram_half_space}
\end{equation}
Indeed, Lemma~\ref{l:second} and \eqref{eq:covariogram_half_space} yield
$$\left.\frac{d}{dr}\vol_n(K_r(\bar\theta))\right|_{r = 0} = -\int_{\s} \max_{1 \le i \le \hid} \langle u, \theta_i\rangle_-\,d\sigma_K(u),$$
and then the claim follows by use of the Gauss map.

Recall that, by definition, $K \cap (K + r \theta_1) \cap \cdots \cap (K + r \theta_\hid) \neq \emptyset$ for $r\in [0,\rho_{D^m(K)}(\bar\theta)]$. Henceforth, we fix such an $r$. Next, we note that for any $\theta \in \mathbb R^n$, $h_{K + r\theta}(u) = h_K(u) + r\langle u, \theta\rangle$. Also, for any convex body $L$, we have $L = \bigcap_{u \in \s} \{x: \langle u, x\rangle \le h_L(u)\}$  \cite[Corollary 1.3.5, pg. 12]{Sh1}. Setting $\theta_0 = o$ for notational convenience, we also have
$$\min_{0 \le i \le \hid} \langle u, \theta_i\rangle = \min_{1 \le i \le \hid} (-\langle u, \theta_i\rangle_-)=- \max_{1 \le i \le \hid} \langle u, \theta_i\rangle_-.$$
Therefore,
\begin{equation}
\label{eq:intersection_formulas}
\begin{split}
K \cap& (K + r \theta_1) \cap \cdots \cap (K + r \theta_\hid) = \bigcap_{i = 0}^\hid \bigcap_{u \in \s} \{x: \langle u, x\rangle \le h_{K + r\theta_i }(u)\} \\
&= \bigcap_{u \in \s} \bigcap_{i = 0}^\hid \{x: \langle u, x\rangle \le h_{K + r\theta_i }(u)\} \\
&= \bigcap_{u \in \s} \left\{x:  \langle u, x\rangle \le \min_{0\le i \le \hid} (h_{K}(u) + r\langle \theta_i, u\rangle)\right\} \\
&= \bigcap_{u \in \s} \left\{x:  \langle u, x\rangle \le h_K(u) + r\min_{0\le i \le \hid} \langle u, \theta_i\rangle\right\}
\\
&= \bigcap_{u \in \s} \left\{x:  \langle u, x\rangle \le h_K(u) - r\max_{0\le i \le \hid} \langle u, \theta_i\rangle_-\right\}
=K_r(\bar\theta).
\end{split}
\end{equation}
Taking volume establishes \eqref{eq:covariogram_half_space}.
 \end{proof}

\noindent We additionally remark that as a function in the variable $r$, $\Cov(r\bar\theta)$ is monotonically decreasing for $r\in[0,\rho_{D^\hid (K)}(\bar\theta)].$ Therefore, from Lebesgue's theorem, $\Cov(r\bar\theta)$ is differentiable almost everywhere on $[0,\rho_{D^\hid (K)}(\bar\theta)]$.

\subsection{Properties of the mth-Order Projection Body}
\label{sec:prop_proj}
In this section, we study properties of the operator $\P$ introduced in Definition~\ref{def:higher_proj}. The formula for the support function of $\P K$ was written as an integral over $\partial K$; by using \eqref{eq:mixed_0} and the definition of the surface area measure, we can also write
\begin{equation}
h_{\P K}(\bar\theta)=nV_n(K[n-1],C_{-\bar\theta})
=\int_{\s}\max_{1\le i \le \hid} \langle u, \theta_i\rangle_{-}d\sigma_K(u).
\label{eq_supp_again}
\end{equation}
It is easily seen that, for any $\bar\theta \in \mathbb R^{n\hid}$, $C_{-\bar\theta} = -C_{\bar\theta}$. As a consequence, it holds
\begin{align*}h_{\P K}(-\bar\theta)&=nV_n(K[n-1],C_{\bar\theta})=\int_{\s}\max_{1\le i \le \hid} \langle u, \theta_i\rangle_{+}d\sigma_K(u).
\end{align*}

Now, we shall introduce a functional associated with $\P$ and show that it is affine-invariant. First, we show how $\P K$ behaves under linear transformations. For $T\in GL_n(\R),$ $\bar x\in \Rnhi,$ we define $\overline{T} \in GL_{n\hid}(\R)$ by $\overline{T}(\bar x)= (T(x_1),\dots,T(x_n)).$
\begin{proposition}
\label{prop:linear:transformations}
    Let $T\in GL_n(\R).$ For $\hid \in \mathbb{N}$ and $K\in\conbod$, one has
    \begin{equation}
        \P TK=|\det T|\overline{T^{-t}}\P K.
    \end{equation}
\end{proposition}
\begin{proof}
    Begin by writing, from Theorem~\ref{t:variationalformula}, that
    \begin{align*}h_{\P TK}(\bar\theta)&=-\lim_{t\to 0}\frac{\vol_n\left((TK)\cap\bigcap_{i=1}^\hid (TK+t\theta_i)\right)-\vol_n(TK)}{t}
    \\
    &=-\lim_{t\to 0}\frac{\vol_n\left(T\left(K\cap\bigcap_{i=1}^\hid (K+tT^{-1}\theta_i)\right)\right)-\vol_n(TK)}{t}
    \\
    &=-|\det T|\lim_{t\to 0}\frac{\vol_n\left(K\cap\bigcap_{i=1}^\hid (K+tT^{-1}\theta_i)\right)-\vol_n(K)}{t}
    \\
    &=|\det T|h_{\P K}(\overline{T^{-1}}\bar\theta)=h_{|\det T|\P K}(\overline{T^{-1}}\bar\theta)
    \\
    &=h_{\overline{T^{-t}}|\det T|\P K}(\bar\theta),
    \end{align*}
    and the claim follows.
\end{proof}
We now show the associated affine invariant quantity for $\PP K.$
\begin{proposition}[Petty Product for mth-Order Projection Bodies]
\label{p:affine_invar_Zhang}
    For $\hid\in\mathbb{N}$ the following functional is invariant under affine transformations:
    $$K\in\conbod \mapsto \vol_n(K)^{n\hid-\hid}\vol_{n\hid}(\PP K).$$
\end{proposition}
\begin{proof}
    Since both volume, as a functional on $\R^n,$ and the
    surface area measure (which define $\P K)$ are translation invariant, it suffices to consider only a linear transformation. Then, the claim is immediate from Proposition~\ref{prop:linear:transformations}. Indeed, taking polarity, one obtains
    $$\PP TK=|\det T|^{-1}\overline{T}\PP K.$$
    Noting that the determinant of $\overline{T},$ considered as a matrix on $\R^{n\hid},$ is equal to $|\det T|^\hid,$ one obtains $\vol_{n\hid}(\PP TK)=|\det T|^{\hid-n\hid}\vol_{n\hid}(\PP K)$. The claim follows.
\end{proof}

Throughout this work, we will use that $\P$ is continuous as an operator on $\conbod$. 

\begin{proposition}[Continuity of $\P$] 
\label{p:continuityofmultdimproj}
    Let $\{K_j\}_j$ be a sequence of convex bodies such that $K_j\to K$ for some $K\in\conbod$ with respect to the Hausdorff metric on $\conbod$. Then, for every $\hid\in \mathbb{N}$, one has $\P K_j\to \P K$ with respect to the Hausdorff metric on $\mathcal{K}^{n\hid}.$
\end{proposition}
\begin{proof} The statement that $\P K_j \to \P K$ in the Hausdorff metric is equivalent to the claim that $h_{\P K_j} \to h_{\P K}$ uniformly on $\S$. One easily verifies that the map $\S \to \conbod$ defined by $\bar\theta \mapsto C_{-\bar\theta}$ is continuous. As we have seen, $h_{\P K}(\bar\theta) = nV_n(K[n - 1], C_{-\bar\theta})$. The mixed volumes are continuous in the Hausdorff metric \cite[p. 280]{Sh1}, so the map $\conbod \times \S \to \mathbb R^+$ defined by $(K, \bar\theta) \to h_{\P K}(\bar\theta)$ is continuous. Since $\S$ is compact, this implies that the map $\conbod \to C(\S)$ defined by $K \mapsto h_{\P K}$ is continuous, which precisely means that if $K_j \to K,$ then $h_{\P K_j} \to h_{\P K}$ uniformly on $\S$.
\end{proof}

In the remainder of this subsection, we study some additional properties of the operator $\P K$. It is well-known that $\Pi K$ uniquely determines $K$ if $K$ is assumed to be origin-symmetric \cite[§10.9]{Sh1}. We first show that for $m \ge 2$, $\P K$ determines $K$ (up to translations) without the symmetry restriction.

\begin{proposition}\label{prop:proj_uniqueness} Fix $m\geq 2$ and let $K, L \in\conbod$. If $\Pi^m K = \Pi^m L$ then $K = L + u$ for some $u \in \R^n$.
\end{proposition}
\begin{proof}Consider the vectors $\bar\theta = (x, y, o, \ldots, o)$: as $x, y$ vary in $\R^n$, $C_{\bar\theta}$ varies over the set of triangles with a vertex at the origin, so the assumption $\P K = \P L$ implies that $V_n(K[n-1], T) = V_n(L[n-1], T)$ for any triangle $T$ with a vertex at the origin, and hence, by translation invariance of mixed volumes, for any triangle. By linearity, it then follows that $V_n(K[n-1], M) = V_n(L[n-1], M)$ for any $M \in \conbod$ which is a Minkowski sum of triangles, and by continuity, when $M$ is a Hausdorff limit of such Minkowski sums (a so-called ``triangle body''). Again by linearity, this implies that $V_n(K[n-1], M) = V_n(L[n-1], M)$ if $M$ is a generalized triangle body, i.e., if there exist triangle bodies $M_1, M_2$ such that $M + M_1 = M_2$. Finally, since the set of generalized triangle bodies is dense in $\conbod$ \cite[Corollary 3.5.12]{Sh1}, we have $V_n(K[n-1], M) = V_n(L[n-1], M)$ for any $M \in \conbod$. This is well-known to imply that $K$ and $L$ are translates of each other \cite[Theorem 8.1.2]{Sh1}.
\end{proof}

 Observe that item 3.) of Proposition~\ref{p:facts} and Theorem~\ref{t:set_con} show that, if $K$ is a $n$-dimensional simplex, then $\P K$ is not origin-symmetric. We characterize the symmetry of $\P K$ in the next proposition.
 
\begin{proposition}
    Fix $m\geq 2$ and let $K\in\conbod$. Then, $\P (-K)=-\P K$, and $-\P K=\P K$ if and only if $K$ is symmetric.
\end{proposition}
\begin{proof}
    Recall that for any $\bar\theta \in \mathbb R^{n\hid}$, we have $C_{-\bar\theta} = -C_{\bar\theta}$. Hence, by the linear invariance of mixed volumes, we have
    \begin{align*}
        h_{\P K}(-\bar\theta)&= nV_n(K[n-1],C_{\bar\theta})=nV_n(-K[n-1],-C_{\bar\theta})
        \\
        &= n V_n(-K[ n- 1], C_{-\bar\theta}) = h_{\P (-K)}(\bar\theta).
    \end{align*}
    This immediately yields $\P (-K)=-\P K$. Since $\P (K + x) = \P K$ for any $x \in \Rn$, it follows more generally that $\P K$ is origin-symmetric if $K$ is symmetric.

     Conversely, if $\P K = -\P K$ then $\P K = \P (-K)$, which implies by Proposition \ref{prop:proj_uniqueness} that $K$ is a translate of $-K$, i.e., $K$ is symmetric.
\end{proof}

Finally, we study $\Pi^2 B_2^2$ as a concrete example; recall the concept of mean width from \eqref{eq:mean_wdith}.

\begin{example}
\label{ex:ball}
When $K=B_2^n$, Definition~\ref{def:higher_proj} becomes
$$h_{\P \B}(\bar{\theta})=n\kappa_n w_n(C_{\bar\theta}).$$

\noindent Consider the simple case when $n=\hid =2.$ Here, 
$$2\pi w_2(L)=2V_2(B_2^2,L)=2V_2(L,B_2^2)=\vol_1(\partial L).$$
\noindent Observe that, for $\bar\theta=(\theta_1,\theta_2)\in\mathbb{S}^{3}$: 
\[
h_{\Pi^2 B_2^2}((\theta_1,\theta_2))=\vol_1(\partial \conv([o,\theta_1],[o,\theta_2]))=|\theta_1|+|\theta_2|+|\theta_1-\theta_2|,
\]
where $\theta_1,\theta_2\in \R^2$ are such that $|\theta_1|^2+|\theta_2|^2=1.$ In particular, this shows that $\Pi^2 B_2^2$ is not a direct product of balls, as $h_{B_2^2\times B_2^2}((\theta_1,\theta_2))=|\theta_1|+|\theta_2|.$ In fact, this shows that $\Pi^2 B_2^2 \supset B_2^2\times B_2^2.$ It is perhaps more natural to compare $\Pi^2 B_2^2$ to $\p B_2^2 \times \p B_2^2 = (2B_2^2) \times (2B_2^2).$ Here,
\[h_{(2B_2^2)\times (2B_2^2)}((\theta_1,\theta_2))=|\theta_1|+|\theta_2| + |\theta_1|+|\theta_2| \geq h_{\Pi^2 B_2^2}((\theta_1,\theta_2))\]
from the triangle-inequality, and so $\Pi^2 B_2^2 \subset (2B_2^2) \times (2B_2^2).$

\noindent One can verify that $\Pi^2 B_2^2$ is, in fact, not of the form $T(K_1\times K_2)$ for $K_1,K_2\in\mathcal{K}^2$ and $T$ an affine transformation of $\R^4$.
\end{example}

\subsection{Properties of the mth-Order Centroid Body}
\label{sec:cent_prop}
We now list properties concerning the $m$-th centroid body operator $\G$, whose definition is given in Definition~\ref{def:mth_centroid}. Our first step is to determine the behavior of $\G L$ under linear transformations.
\begin{proposition}
\label{p:cen_tran}
     Let $T\in GL_n(\R).$ For $\hid \in \mathbb{N}$ and a compact set $L\subset \Rnhi$ with positive volume, one has
    \begin{equation}
        \G  \overline{T} L = T\G  L. 
        \label{trans_cent}
    \end{equation}
\end{proposition}
\begin{proof}
    The result follows from the definition. Indeed, from \eqref{eq:cen_hi}, one has
    \begin{align*}
    h_{\G  \overline{T} L}(\theta)&=\frac{1}{\vol_{n\hid}(\overline{T}L)} \int_{\overline{T}L} \max_{1\leq i \leq \hid} \langle x_i, \theta\rangle_- d\bar{x}
    \\
    &=\frac{1}{\vol_{n\hid}(L)} \int_{L} \max_{1\leq i \leq \hid} \langle Tx_i, \theta\rangle_- d\bar{x}
    \\
    &=\frac{1}{\vol_{n\hid}(L)} \int_{L} \max_{1\leq i \leq \hid} \langle x_i, T^t\theta\rangle_- d\bar{x}
    \\
    &=h_{\G  L}(T^t\theta)=h_{T\G  L}(\theta).
    \end{align*}
\end{proof}

Interestingly, $\G L$ is actually a smooth convex body in $\R^n$.
\begin{proposition}
	       \label{p:GammaC1}
            For compact $L\subset\Rnhi$ with positive volume, $\G L$ is a $C^1$ convex body in $\Rn$.
        \end{proposition}
        \begin{proof}
            Denote $f(\bar x,\xi) = \max_{1\leq i \leq \hid} \langle x_i, \xi \rangle_-$, where $\bar x=(x_1,\ldots,x_\hid)\in\Rnhi$ and $\xi\in\Rn$. Let $\partial_2 f(\bar x,\xi)$ be the vector in $\R^n$ given by the gradient of the function $\xi \mapsto f(\bar x,\xi)$. Let $R$ be the radius of a ball containing $L$.
            We claim that $f$ has the following properties:
            \begin{enumerate}
                \item[(a).] $|f(\bar x,\xi)-f(\bar x, \eta)| \leq (R+1) |\xi - \eta|$ for $|\bar x| \leq R+1$ and all $\xi,\eta \in \Rn$.
                \item[(b).] $|f(\bar x,\xi)| \leq (R+1)|\xi|$ for $|\bar x| \leq R+1$ and all $\xi \in \Rn$.
                \item[(c).] 
                %For $\xi \neq 0$, $f(\cdot, \xi)$ is $C^1$ in a $n\hid$-dimensional, full-measure open set $A(\xi) \subseteq \Rnhi$.
                For every $\xi \neq o$, there is an open set $A(\xi) \subseteq \Rnhi$ satisfying $\vol_{nm}(\R^{nm}\setminus A(\xi))=0$ and, if $\bar x \in A(\xi)$, then the function $f(\bar x, \cdot)$ is $C^1$ in a neighborhood of $\xi$.
                \item[(d).] $|\partial_2 f(\bar x,\xi)| \leq (R+1)$ for $|\bar x| \leq R+1$, $\bar x \in A(\xi)$, $\xi \neq o$.

            \end{enumerate}
            Indeed, for item (a), write $f(\bar x,\xi)=h_{C_{-\bar x}}(\xi)$. Then, we use that $h_{C_{-\bar x}}$ is sublinear to obtain 
            \[
            h_{C_{-\bar x}}(\xi)-h_{C_{-\bar x}}(\eta)\leq h_{C_{-\bar x}}(\xi-\eta).
            \]
            We introduce the notation $x_\xi$ for any element of $ \{x_1,\dots,x_m\}$ where the maximum in the definition of $h_{C_{-\bar x}}(\xi)$ is obtained for $\xi$, i.e. $h_{C_{-\bar x}}(\xi)=\langle x_\xi,\xi\rangle_-$. Then, by the Cauchy-Schwarz inequality,
            \[
            h_{C_{-\bar x}}(\xi-\eta) = \langle x_{\eta-\xi}, \xi -\eta\rangle_- \leq |\langle x_{\eta-\xi}, \xi -\eta\rangle|\leq |x_{\eta-\xi}||\eta-\xi|\leq (R+1)|\eta-\xi|.
            \]
            By permuting $\eta$ and $\xi$, the claim follows. Item $(b)$ follows from $(a)$ by picking $\eta=o$.

            For part (c), first note that the function $\xi \mapsto \langle x, \xi\rangle_-$ is $C^1$ except on the subspace $\{x\in\R^n:\langle x, \xi\rangle = o\}$, and that if $f, g$ are $C^1$ on some neighborhood then $\max(f, g)$ is $C^1$ away from the closed set $\{f = g\}$, as for any $x$ such that $f(x) > g(x)$ there exists a neighborhood $V$ such that $\left.\max(f, g)\right|_V = \left.f\right|_V$, and similarly if $g(x) > f(x)$. Hence $\xi \mapsto \max_i \langle x_i, \xi\rangle_-$ is $C^1$ in a neighborhood of $\xi$ for a fixed $\bar x$ satisfying $\xi\mapsto h_{C_{-\bar x}}(\xi)$ is differentiable. We then define the set $A(\xi)\subset \R^{nm}$ to be the set of all $\bar x$ with this property; it is defined as the complement of the set
            $$\bigcup_{i = 1}^m \{\bar x\in\R^{nm}:\langle x_i, \xi\rangle = o\} \cup \bigcup_{\substack{i,j\in\{1,\ldots,m\}\\i\neq j}} \{\bar x\in\R^{nm}:\langle x_i, \xi\rangle = \langle x_j, \xi\rangle\},$$
            which, in turn, is a union of linear subspaces and hence a closed set of Lebesgue measure $0$.

            For $(d)$, we use that $h_{C_{-\bar x}}$ is $C^1$ near $\xi$ when $\xi\neq o$ and $\bar x\in A(\xi)$ to write, for $t>0$ small and $u\in\s$,
            \[
            h_{C_{-\bar x}}(\xi+tu)=h_{C_{-\bar x}}(\xi)+t\langle \nabla h_{C_{-\bar x}}(\xi),u \rangle +o(t)
            \]
            where $o$ is a function so that $o(t)/t \to 0$ as $t\to 0$. Therefore, from item (a),
            \[
            \left|\langle \nabla h_{C_{-\bar x}}(\xi),u \rangle +\frac{o(t)}{t}\right| = \left|\frac{h_{C_{-\bar x}}(\xi+tu)-h_{C_{-\bar x}}(\xi)}{t}\right| \leq (R+1).
            \]
            Sending $t\to 0$, we obtain
            \[
            |\langle \nabla h_{C_{-\bar x}}(\xi),u\rangle| \leq R+1.
            \]
            Taking supremum over all $u \in \s$, we obtain that $|\nabla h_{C_{-\bar x}}(\xi)| \leq R+1$.

            Now, we are ready to prove $\Gamma^m L$ is a $C^1$ convex body. It suffices to show that $h_{\G L}$ is $C^1$ outside the origin.
            From (a) we deduce immediately that $F(\xi)= \vol_{nm}(L)h_{\G L}(\xi) =\int_L f(\bar x,\xi)d\bar x$ is continuous in $\Rn$. Fix $\xi \neq o$ and set $Q(\xi) = \int_L \partial_2 f(\bar x,\xi) d\bar x$ (which is well-defined by $(c)$). For any sequence $\eta_j \to o$ in $ \Rn$,
            \begin{align*}\big| F(\xi+\eta_j) - F(\xi) &-\langle Q(\xi),\eta_j\rangle\big| \leq
            \\
            &\int_L |f(\bar x,\xi+\eta_j) - f(\bar x,\xi) - \langle\partial_2 f(\bar x,\xi),\eta_j\rangle| d\bar x.\end{align*}
            By $(c)$, the integrand tends to $0$ for a.e $\bar x \in \Rnhi$. By $(b)$ and $(d)$, we may apply the dominated convergence theorem to deduce that the last integral tends to $0$. Since this happens for every sequence $\eta_j,$ we deduce that $F$ is differentiable at $\xi$ and its differential is $Q(\xi)$.

            Finally, let $\xi_j \to \xi$. We compute:
            \[|Q(\xi_j) - Q(\xi)| \leq \int_L |\partial_2 f(\bar x, \xi_j) - \partial_2 f(\bar x, \xi)|d \bar x.\]

            Again by $(c)$, the integrand tends to $0$ for a.e. $\bar x \in \Rnhi$. By $(d)$, we may apply the dominated convergence theorem and we obtain that $Q(\xi_j) \to Q(\xi)$ showing that $F$ is $C^1$.
        \end{proof}

    The next lemma shows that $\G$ and $\PP$ are connected through mixed volumes.
\begin{lemma}
        \label{lem_duality}
        For $K\in\conbod$ and $L \in \sta,$ one has
        \[ \widetilde V_{-1}(L[n\hid+1], \PP K) = \vol_{n\hid}(L) \frac{n\hid+1}{\hid} V_n(K[n-1], \G  L),\]
where the dual mixed volume $\widetilde V_{-1}(\cdot,\cdot)$ is $n\hid$-dimensional and the mixed volume $V_n(\cdot,\cdot)$ is $n$-dimensional.
\end{lemma}
\begin{proof}
From the definition of the dual mixed volume and mixed volume, we have
        \begin{align}
                \widetilde V_{-1}&(L[n\hid+1], \PP K)
                \\
                &= \frac 1{n\hid} \int_{\S} \rho_L(\bar{\theta})^{n\hid+1} \int_{\s} \max_{1\leq i \leq \hid} \langle \theta_i, \xi \rangle_- d \sigma_K(\xi) d \bar{\theta}\\
                &= \frac 1{n\hid} \int_{\s} \int_{\S} \rho_L(\bar{\theta})^{n\hid+1} \max_{1\leq i \leq \hid} \langle \theta_i, \xi \rangle_- d \bar{\theta} d \sigma_K(\xi)\\
                &= \frac {n\hid+1}{n\hid} \int_{\s} \int_{\S} \int_0^{\rho_L(\bar{\theta})} t^{n\hid-1} \max_{1\leq i \leq \hid} \langle t \theta_i, \xi \rangle_- dt d \bar{\theta} d \sigma_K(\xi)\\
                &= \frac {n\hid+1}{n\hid} \int_{\s} \int_L \max_{1\leq i \leq \hid} \langle x_i,  \xi \rangle_- d\bar{x} d \sigma_K(\xi)\\
                &= \vol_{n\hid}(L)\frac {n\hid+1}{n\hid} \int_{\s} h_{\G  L}(\xi) d \sigma_K(\xi)\\
                &= \vol_{n\hid}(L)\frac{n\hid+1}\hid V_n(K[n-1], \G  L).
        \end{align}
\end{proof}

We conclude this section by showing an interaction of $\G$ and $\PP$ when applied to ellipsoids. Notice that $\PP:\conbod \to \Conbod$ and $\G:\Conbod \to \conbod$.

\begin{lemma}
\label{l:class}
Let $E$ be a centered ellipsoid in $\R^n.$ Then, $$\G  \PP E= \frac{\hid}{n\hid+1} \frac{1}{\vol_n(E)} E.$$
\end{lemma}
\begin{proof}
First, we show that $\G\PP \B$ is rotation invariant, and thus a dilate of $\B.$ We first notice that, for every $T\in \mathcal{O}(n),$ \eqref{trans_cent} and Proposition~\ref{prop:linear:transformations} yield $\G \overline{T} L= T \G L$ and $\PP T K = \overline{T} \PP K$ for $L\in\Conbod$ and $K\in\conbod.$ Thus, for $K=\B$ and $L=\PP \B,$ one has
 \[\G\PP \B = \G\PP T \B = \G \overline{T} \PP \B = T \G\PP \B;\]
this means that $\G\PP \B$ is rotation invariant and thus a ball. Next, let $E\in\conbod$ be a centered ellipsoid. Then, there exists $T\in GL_n(\R)$ so that $E=T \B.$ From Proposition~\ref{prop:linear:transformations}, we have
     $$\PP T \B=\overline T|\det T|^{-1}\PP \B.$$
     Applying $\G,$ we obtain
     \begin{align*}
     \G\PP T \B
     &=\G\overline T|\det T|^{-1}\PP \B \\
     &=|\det T|^{-1}T \G\PP \B \\
     &=\kappa_n\vol_{n}(E)^{-1}C_{n,\hid} T \B \\
     &=\kappa_n\vol_{n}(E)^{-1} C_{n,\hid} E
     \end{align*}
     for some $C_{n,\hid}>0.$  To establish the formula for $C_{n,m}$, set $K=\B$ and $L= \PP \B$ in in Lemma~\ref{lem_duality} to obtain
   \begin{align*}
		\frac{m}{nm+1} &=  V_n(\B[n-1], \G \PP \B) 
  = V_n(\B [n-1], C_{n,m} \B)=\kappa_nC_{n,m}.
    \end{align*}
\end{proof}

\section{The mth-order Radial Mean Bodies}
\label{sec:new_rad}
The following is in \cite[Section 6]{FLM20}. Firstly, let $\psi:[0,\infty)\to[0,\infty)$ be an integrable function that is right continuous and $s$-concave for some $s>0$. Then, its Mellin transform is the map given by
\begin{equation} \mathcal{M}_\psi :p \mapsto
    \begin{cases}
    \int_0^\infty t^{p-1}(\psi(t)-\psi(0))dt, & p\in (-1,0),
    \\
    \int_0^\infty t^{p-1}\psi(t)dt, & p>0,
    \end{cases}
    \label{eq:Mel}
\end{equation}
When viewed as a function on $\mathbb{C}$, the Mellin transform is analytic on the half-plane $\{p\in \mathbb{C}:\mathrm{Re}(p)>0\}$. If $\psi$ is additionally right-differentiable at zero, then its Mellin transform $\mathcal{M}_\psi$ is a meromorphic function on $\{p\in \mathbb{C}:\mathrm{Re}(p)>-1\}$ with a simple pole at the origin. This is verified via integration by parts: for $p\in (-1,0)\cup(0,+\infty),$ one has
\begin{equation}
\label{eq:mellin_integration}
\mathcal{M}_\psi(p) = \frac{1}{p}\int_{0}^\infty t^p (-\psi(t))^\prime dt = \int_0^1 t^{p-1}(\psi(t)-\psi(0))dt +\frac{\psi(0)}{p} + \int_1^\infty t^{p-1}\psi(t)dt.\end{equation}

For $s>0$, let $\psi_s(t)=(1-t)_+^{1/s}$. Then, we have the following monotonicity result.

\begin{lemma}[The Mellin-Berwald inequality, Theorem 6.1 from \cite{FLM20}] 
\label{l:mono}
Let $\psi$ be a non-increasing, $s$-concave function, $s>0$. Then the function
\begin{equation}G_{\psi}(p):=\left(\frac{\mathcal{M}_\psi(p)}{\mathcal{M}_{\psi_s}(p)}\right)^{1/p}=\left(p\binom{p+\frac{1}{s}}{p}\mathcal{M}_\psi(p)\right)^{1/p}
\label{eq:Fradelizi}
\end{equation}
is decreasing on $(-1,\infty)$. Additionally, if there is equality for any two $p,q\in(-1,\infty),$ then $G_\psi(p)$ is constant. Furthermore, $G_\psi (p)$ is constant if and only if $\psi^s$ is affine on its support.
\end{lemma}
The equality conditions of Lemma~\ref{l:mono} are not stated in \cite[Theorem 6.1]{FLM20} but implied by the proof; see \cite[Lemma 2.2]{LP25} for an explicit proof. We remark that the Mellin transform is defined, and the Mellin-Berwald inequality holds, for all $s$-concave functions, $s\in \R,$ but, for $s<0,$ the constants become more complicated and one must restrict to $p<-\frac{1}{s}$ for integrability. 

We now introduce the $m$th-order radial mean bodies.

\begin{definition}
\label{def:radial_new}
Let $K\in\conbod$. For $\hid\in\N$ and $p >-1,$ we define the \textit{$(\hid,p)$ radial mean bodies} of $K$, $R^\hid_p K,$ to be the star bodies in $\R^{n\hid}$ whose radial functions are given by, for $\bar{\theta}\in\S$:
\begin{equation}
    \rho_{R_p^{\hid} K}(\bar\theta)=\begin{cases}\left(\frac{1}{\vol_n(K)} \int_K \left(\min _{i=1, \ldots, \hid}\left\{\rho_{K-x}(-\theta_i)\right\}\right)^p d x\right)^\frac{1}{p}, & p>-1,p\neq 0;
    \\
    \exp\left(\frac{1}{\vol_n(K)}\int_K\log\left(\min _{i=1, \ldots, \hid}\left\{\rho_{K-x}(-\theta_i)\right\}\right)dx\right),  & p=0;
    \\
    \max_{x\in K}\min_{i=1,\dots,\hid}\rho_{K-x}(-\theta_i), & p=\infty.
    \end{cases}
\end{equation}
\end{definition}
It follows from Proposition~\ref{p:other_form} below that $R^\hid_p K$ is a star body for all $p>-1$. Notice that $R^1_p K=R_p K$, since, when $m=1$, Definition~\ref{def:radial_new} recovers \eqref{pth} when one uses the fact that the bodies $R_p K$ are origin-symmetric. We now show that $R^\hid_\infty K=D^\hid(K).$ 
\begin{proposition}
    Fix $\hid,n\in \N$ and $K\in\conbod.$ Then, $R^\hid_\infty K=D^\hid(K)$ and $R^\hid_p K \to \{o\}$ as $p\to (-1)^+$.
\end{proposition}
\begin{proof}
Since $R^\hid_\infty K$ is a star body, we know that
$$R^\hid_\infty K=\{\bar y\in\Rnhi:\rho_{R^\hid_\infty K}(\bar y) \geq 1\}.$$
Hence, if $\bar y\in R^\hid_\infty K$, then, $\max_{x\in K}\min_{i}\rho_{K-x}(-y_i) \geq 1$; but this is true if and only if there exists at least one $x\in K$ such that $\min_{i=1,\dots,\hid}\rho_{K-x}(-y_i) \geq 1.$ But, for such an $x,$ this means, for every $i=1,\dots,\hid$, one has $\rho_{K-x}(-y_i) \geq 1.$ However, $K-x$ is a star body, and hence this inequality implies $-y_i\in K-x,$ or, equivalently, $x\in K+y_i.$ Therefore, we have $$x\in K\cap \bigcap_{i=1}^\hid (K+y_i).$$
By definition, this means that $\bar y\in D^\hid(K).$ So, $R_\infty^\hid K \subseteq D^\hid(K).$ Conversely, if we know that $\bar y \in D^\hid (K),$ then there exists some $x\in K\cap \bigcap_{i=1}^\hid (K+y_i).$ This means that $\rho_{K-x}(-y_i) \geq 1$ for all $i=1,\dots,\hid.$ Thus, so, too is their minimum. Taking maximum over $x\in K$ yields $\max_{x\in K}\min_{i}\rho_{K-x}(-y_i) \geq 1,$ i.e. $\bar y \in R^\hid_\infty K,$ and so the claim follows.

For the second claim, we note that, when $p\in (-1,0),$
\[
\int_K \left(\min _{i=1, \ldots, \hid}\left\{\rho_{K-x}(-\theta_i)\right\}\right)^p d x \geq \int_K \rho_{K-x}(-\theta_i)^p d x
\]
for all $i=1,\dots,m$, and the latter integral was shown to go to $\infty$ as $p\to -1$ by Gardner and Zhang \cite{GZ98}.
\end{proof}

\noindent We now show that the radial functions $\rho_{R_p^m K}$ of the $(\hid,p)$ radial mean bodies $R_p^m K$ are actually the Mellin transform of the $\hid$th-covariogram $\Cov$. In the presence of $(-\Cov(r\bar\theta))^\prime$ below, the derivative is in $r$, and, we make note of the fact that, since $\Cov(r\bar\theta)$ is decreasing, $-\Cov(r\bar\theta)^\prime$ is positive.
\begin{proposition}
    Let $\hid,n\in\N$ be fixed. Then, for $K\in\conbod,$ one has that, for $p\neq 0$,
    \begin{equation}
\begin{split}
\rho_{R^\hid_p K}(\bar{\theta})&=\left(p\mathcal{M}_{\frac{\Cov(r\bar\theta)}{\vol_n(K)}}(p)\right)^\frac{1}{p}
\\
&=\begin{cases}
    \left(\frac{p}{\vol_n(K)}\int_0^{\rho_{D^\hid(K)}(\bar\theta)}\Cov(r\bar{\theta})r^{p-1}dr\right)^\frac{1}{p}, & p>0,
    \\
    \left(p\int_{0}^{\rho_{D^\hid (K)}(\bar{\theta})} \left(\frac{\Cov(r\bar{\theta})}{\vol_n(K)}-1\right) r^{p-1} d r +\rho^p_{D^\hid (K)}(\bar{\theta})\right)^\frac{1}{p}, & p\in (-1,0).
\end{cases}
\end{split}
\label{radial_ell}
\end{equation}
From Proposition~\ref{p:radial_ball}, this yields that, for each $p \geq 0$, $R^\hid_{p} K\in\mathcal{K}^{nm}$ and that it contains the origin in its interior. Additionally, the above formula re-writes as, for $p \in (-1,\infty)$:
\begin{equation}\rho_{R^\hid_p K}(\bar{\theta})=\begin{cases}\left(\frac{1}{\vol_n(K)}\int_0^{\rho_{D^\hid (K)}(\bar\theta)}(-\Cov(r\bar\theta))^\prime r^{p}dr\right)^\frac{1}{p}, & p\neq 0,
\\
\exp\left(\frac{1}{\vol_n(K)}\int_0^{\rho_{D^m(K)}(\bar\theta)}(-\Cov(r\bar\theta))^\prime\log(r)dr\right), & p=0.
\end{cases}
\label{eq:beset_radial_form}
\end{equation}
\label{p:other_form}
\end{proposition}
\begin{proof}
Observe that $x\in \cap_{i=1}^\hid(K+r\theta_i),$ if and only if $-r\theta_i\in K-x$ for every $i=1,\dots,\hid.$ But this is equivalent to $0\leq r\leq \rho_{K-x}(-\theta_i).$ Then, for $p>0$:
\begin{align*}
    \int_K&\left(\min_{i=1,\dots,\hid}\{\rho_{K-x}(-\theta_i)\}\right)^pdx =p\int_K\int_0^{\min_{i=1,\dots,\hid}\{\rho_{K-x}(-\theta_i)\}}r^{p-1}drdx
    \\
    &=p\int_{K}\int_0^\infty \chi_{\bigcap_{i=1}^\hid\{r>0:r\leq \rho_{K-x}(-\theta_i)\}}(r)r^{p-1}drdx
    \\
    &=p\int_0^\infty \int_{K}\chi_{\bigcap_{i=1}^\hid\{x\in\Rn:x\in (K+r\theta_i)\}}(x)dx\, r^{p-1}dr
    \\
    &=p\int_0^\infty \Cov(r\bar\theta)r^{p-1}dr=p\int_0^{\rho_{D^\hid(K)}(\bar\theta)} \Cov(r\bar\theta)r^{p-1}dr.
\end{align*}
For $p\in (-1,0),$ we obtain
\begin{align*}
    \int_K&\left(\min_{i=1,\dots,\hid}\{\rho_{K-x}(-\theta_i)\}\right)^pdx =-p\int_K\int_{\min_{i=1,\dots,\hid}\{\rho_{K-x}(-\theta_i)\}}^\infty r^{p-1}drdx
    \\
    &=-p\int_{K}\int_0^\infty \chi_{\bigcup_{i=1}^\hid\{r>0:r> \rho_{K-x}(-\theta_i)\}}(r)r^{p-1}drdx
    \\
    &=-p\int_0^\infty \int_{K}\chi_{\bigcup_{i=1}^\hid\{x\in\Rn:x\notin (K+r\theta_i)\}}(x)dxr^{p-1}dr
     \\
    &=-p\int_0^\infty \vol_n\left(K\setminus \left(\bigcap_{i=1}^\hid (K+r\theta_i)\right)\right)r^{p-1}dr
    \\
    &=-p\int_0^\infty (\vol_n(K)-\Cov(r\bar\theta))r^{p-1}dr
    \\
    &=p\int_0^{\rho_{D^\hid(K)}(\bar\theta)} (\Cov(r\bar\theta)-\vol_n(K))r^{p-1}dr +\vol_n(K)\rho^p_{D^\hid (K)}(\bar{\theta}).
\end{align*}
The equation~\eqref{eq:beset_radial_form} comes from \eqref{eq:mellin_integration} and the differentiability of $\Cov(r\bar\theta)$ almost everywhere (as a function in $r$ on its support) for $p\neq 0$, and then via continuity for $p=0$.
\end{proof}
The convexity for $p \in (-1,0)$ is unknown. On the other-hand, for $p\in (-1,0),$ let $s:=1+p\in (0,1).$
Then, we obtain
\begin{equation}s\rho_{R^\hid_p K}(\bar{\theta})^{s-1}=s\int_0^{\rho_{D^\hid (K)}(\bar\theta)}\frac{(-\Cov(r\bar\theta))^\prime}{\vol_n(K)} r^{s-1}dr.
\label{eq_almost-1}
\end{equation}
We recall the following elementary lemma; see, for example, \cite[Lemma 4]{HL25} for a proof.
\begin{lemma}
\label{l:fractional_deriv}
If $\varphi:[0, \infty) \rightarrow[0, \infty)$ is a measurable function such that, for some $s_0 \in(0,1)$, one has $\int_0^{\infty} t^{s_0-1} \varphi(t) \mathrm{d} t<\infty$, then
$$
\lim _{s \rightarrow 0^{+}}s \int_0^{\infty} t^{s-1} \varphi(t) \mathrm{d} t=\lim _{t \rightarrow 0^{+}} \varphi(t).
$$
\end{lemma}

\noindent Consequently, we can now establish the behavior of $\rho_{R^\hid_p K}$ as $p\to (-1)^+.$ 
\begin{proposition}
    Fix $\hid,n\in\N$ and $K\in\conbod.$ Then, $$\lim_{p\to (-1)^+}(p+1)^{1/p}R^\hid_p K = \vol_n(K) \PP K.$$
\end{proposition}
\begin{proof}
Taking the limit of \eqref{eq_almost-1}, using Lemma~\ref{l:fractional_deriv}, and inserting the result for the derivative of the covariogram from Theorem~\ref{t:variationalformula},
we obtain  for every $\bar\theta\in\S$ that
\begin{equation}
\label{eq:-1_body}
\lim_{p\to (-1)^+}(p+1)^{1/p}\rho_{R^\hid_p K}(\bar\theta)=\vol_n(K) \rho_{\PP K}(\bar\theta).\end{equation}
\end{proof}
Thus, we see that the \textit{shape} of $R^{\hid}_p K$ approaches that of $\vol_n(K) \PP K$ as $p\to (-1)^+.$ With the aid of \eqref{eq:beset_radial_form}, one can deduce the following chain of set-inclusions.
\begin{lemma}
    Fix $n,\hid\in\N.$ For every $K\in\conbod$ and $-1<p < q <\infty$ one has:
$$R^\hid_pK \subseteq R^\hid_qK \subseteq D^\hid(K).$$
\end{lemma}
\begin{proof}
It suffices to show for every $\bar\theta\in\S$ that $\rho_{R^\hid_pK}(\bar\theta)$ is an increasing function in $p$ for $p>-1.$ Indeed, this follows from Jensen's inequality applied to \eqref{eq:beset_radial_form} with respect to the probability measure $$-\vol_n(K)^{-1}\Cov(r\bar\theta)^\prime\chi_{[0,\rho_{D^{\hid}(K)}(\bar\theta)]}(r).$$
\end{proof}
We now show the reverse of this chain of inclusions, that is, we prove Theorem~\ref{t:set_con}. For the convenience of the reader, we shall restate this theorem below.

\begin{reptheorem}{t:set_con}
Let $K\in\conbod$ and $\hid\in\N.$ Then, for $-1< p< q < \infty$, one has
$$D^\hid (K) \subseteq {\binom{q+n}{n}}^{\frac{1}{q}} R^\hid_{q}K \subseteq {\binom{p+n}{n}}^{\frac{1}{p}} R^\hid_{p} K\subseteq n\vol_n(K)\PP K.$$
Equality occurs in any set inclusion if and only if $K$ is a $n$-dimensional simplex.
\end{reptheorem}
\begin{proof}
Recall from Proposition~\ref{p:facts} that $\Cov(r\bar\theta)^{\frac{1}{n}}$ is a concave function (in $r$) on its support for every fixed $\bar\theta\in\S.$ Thus, from Lemma~\ref{l:mono} and Proposition~\ref{p:other_form},
the function $$G_K(p;\bar{\theta}):={\binom{p+n}{n}}^{\frac{1}{p}}\rho_{R^\hid_p K}(\bar{\theta})$$
is non-increasing in $p,$ $p>-1,$ for every fixed $\bar{\theta}$, which establishes the first three set inclusions upon insertion of definitions. For the last set inclusion, we have not yet established the behavior of $\lim_{p\to (-1)^+}G_K(p;\bar{\theta}).$ Observe that $$G_K(p;\bar{\theta})=\left({\frac{1}{p+1}\binom{p+n}{n}}\right)^{\frac{1}{p}}(p+1)^{1/p}\rho_{R^\hid_p K}(\bar{\theta}).$$
From the fact that
$$\left({\frac{1}{p+1}\binom{p+n}{n}}\right)^{\frac{1}{p}} \xrightarrow[p\to (-1)^+]{} n,$$
one obtains from \eqref{eq:-1_body} that
$\lim_{p\to (-1)^+}G_K(p;\bar{\theta}) = n\vol_n(K) \rho_{\PP K}(\bar\theta).$
The equality conditions follow from those of Lemma~\ref{l:mono}, which shows equality occurs for any set inclusion, and hence for all set inclusions, if and only if $\Cov(r\bar\theta)^{\frac{1}{n}}$ is as affine function in $r$ on its support. Indeed, Proposition~\ref{p:simp} below shows this characterizes a simplex.
\end{proof}
The following proposition characterizes a simplex; the equivalence between $(i)$ and $(ii)$ can be found in \cite[Section 6]{EGK64}, or \cite{Choquet,RS57}, and the equivalence between $(ii)$ and $(iii)$ is the content of \cite[Section 5]{Sch70}.
\begin{proposition}
	\label{p:simp}
	Let $K\in\conbod$ and $\hid\in\mathbb{N}.$ The following are equivalent:
	\begin{enumerate}
        \item[(i).] $K$ is a $n$-dimensional simplex.
	    \item[(ii).] \label{item:hom} For every $\theta\in\s$ and $r>0$ so that $K\cap (K+r\theta)\neq\emptyset$, $K\cap (K+r\theta)$ is homothetic to $K$.
	    \item[(iii).] \label{item:aff} For every $\bar\theta\in\S$, $\Cov(r\bar\theta)^{1/n}$ is an affine function in $r$ for $r\in[0,\rho_{D^\hid (K)}(\bar\theta)]$.
	\end{enumerate}
	\end{proposition}

We next show that there exists an $(\hid,p)$ radial mean body whose measure is comparable to that of $K.$
\begin{proposition}
\label{radial_com}
    Let $K\in\conbod$ and $\hid\in\mathbb{N}.$ Then,
    $$\vol_{n\hid}(R^\hid_{n\hid} K)=\vol_n(K)^{\hid}.$$
\end{proposition}
\begin{proof}
We first observe that, from the translation invariance of the Lebesgue measure:
\begin{align*}
    \vol_n(K)^\hid&=\frac{1}{\vol_n(K)}\int_K\prod_{i=1}^\hid\vol_{n}(y-K)dy
    \\
    &=\frac{1}{\vol_n(K)}\int_K\prod_{i=1}^\hid\left(\int_{\R^n}\chi_{y-K}(x_i)dx_i\right)dy
    \\
    &=\frac{1}{\vol_n(K)}\int_{\R^n}\cdots\int_{\R^n}\int_K\chi_{\cap_{i=1}^\hid (x_i+K)}(y)dyd{x_1}\cdots d{x_\hid}
    \\
    &=\frac{1}{\vol_n(K)}\int_{\R^{n\hid}}\Cov(\bar x)d\bar{x} =\frac{1}{\vol_n(K)}\int_{D^\hid(K)}\Cov(\bar x)d\bar x.
    \end{align*}

On the other-hand, integrating in polar coordinates \eqref{eq:polar} and using the polar formula for volume \eqref{eq:polar_2} yields
\begin{align*}\frac{1}{\vol_n(K)}&\int_{D^\hid(K)}\Cov(\bar x)d\bar x
\\
&=\frac{1}{\vol_n(K)}\int_{\S}\!\int_0^{\rho_{D^\hid(K)}(\bar\theta)}\!\Cov(r\bar \theta)r^{n\hid-1}drd\bar\theta
\\
&=\frac{1}{n\hid}\int_{\S}\rho_{R^\hid_{n\hid} K}(\bar\theta)^{n\hid}d\bar\theta=\vol_{n\hid}(R^\hid_{n\hid} K).
\end{align*}
\end{proof}
From this we can prove the following results; the first one is a new proof of \eqref{eq:RSell}.

\begin{corollary}[Rogers-Shephard inequality for mth-order difference bodies]\label{cor:RSInequalityVol}
Fix $\hid \in \N$ and $K\in\conbod.$ Then, one has 
\[
    \vol_n(K)^{-\hid}\vol_{n\hid}\left(D^\hid (K)\right)\leq \binom{n\hid +n}{n},
\] 
with equality if and only if $K$ is a $n$-dimensional simplex.
\end{corollary}

\noindent The second one is the $m$th-order Zhang's projection inequality. We restate it here for convenience.

\begin{repcorollary}{cor:zhanginequality}[Zhang's projection inequality for mth-order projection bodies]
Fix $\hid \in \N$ and $K\in\conbod.$ Then, one has 
\[
\vol_n(K)^{n\hid-\hid}\vol_{n\hid}\left( \PP K \right) \geq \frac{1}{n^{n\hid}}\binom{n\hid+n}{n},
\] 
with equality if and only if $K$ is a $n$-dimensional simplex.
\end{repcorollary}
\begin{proof}[Proofs of Corollaries \ref{cor:zhanginequality} and \ref{cor:RSInequalityVol}]
    From the homogeneity of volume, Theorem \ref{t:set_con} with $p=n\hid,$ and Proposition~\ref{radial_com}, one has
    \begin{align*}\vol_{n\hid}&(D^\hid (K)) \leq \vol_{n\hid}\left({\binom{n\hid+n}{n}}^{\frac{1}{n\hid}} R^\hid_{n\hid}K\right)
    \\
    &=\binom{n\hid+n}{n}\vol_{n\hid}(R^\hid_p K) =\binom{n\hid+n}{n}\vol_n(K)^{\hid}
    \\
    &\leq \vol_{n\hid}\left(n\vol_n(K)\PP K\right)
    \\
    &=n^{n\hid}\vol_n(K)^{n\hid}\vol_{n\hid}(\PP K).\end{align*}
    The equality conditions are immediate.
\end{proof}

\section{The mth-Order Petty's Projection Inequality}
\label{sec:petty}
In this section, we set out to prove Theorem~\ref{t:pettyprojectioninequality}, which we restate for the reader's convenience here.

\begin{reptheorem}{t:pettyprojectioninequality}
Fix any $\hid \in \N$ and $K \in \conbod$. Then, one has 
\[
\vol_n(K)^{n\hid-\hid}\vol_{n\hid}(\PP K) \leq \vol_{n}(\B)^{n\hid-\hid}\vol_{n\hid}(\PP \B),
\]
with equality if and only if $K$ is an ellipsoid.
\end{reptheorem}
Before proceeding to the proof of Theorem~\ref{t:pettyprojectioninequality}, we establish the following class-reduction argument in the spirit of \cite{LYZ00}.

\begin{lemma}
	\label{l:classreduction}
	Fix $\hid,n\in\N$ and let $K \in \conbod$. Then,
	\begin{equation}
            \begin{split}
                \vol_{n\hid}(\PP K)&\vol_{n}(K)^{n\hid-\hid} \leq 
                \\
                &\vol_{n\hid}(\PP \G \PP K) \vol_{n}(\G \PP K)^{n\hid-\hid}, 
            \end{split}
                \label{eq_duality_4}
	\end{equation}
	with equality if and only if $K$ is homothetic to $\G\PP K$. 
\end{lemma}
\begin{proof}
    By setting $K=\G L$ in Lemma \ref{lem_duality} and using the dual Minkowski's first inequality \eqref{dual_Min_first}, we obtain for every $L\in\sta$ that:
\begin{equation}
        \label{eq_duality_2}
        \vol_n(\G L)^{n\hid} \geq \left(\frac{n\hid+1}{\hid}\right)^{-n\hid} \vol_{n\hid}(L) \vol_{n\hid}(\PP \G L)^{-1}
\end{equation}
    with equality if and only if $L$ is a dilate of $\PP \G L$. Then, set $L = \PP K$ in \eqref{eq_duality_2} to obtain for every $K\in\conbod$ that
        \begin{equation}
        \begin{split}
                \vol_{n}(\G \PP K)^{n\hid}&\vol_{n\hid}(\PP \G \PP K) 
                \\
                &\geq \left(\frac{n\hid+1}{\hid}\right)^{-n\hid}\vol_{n\hid}(\PP{K}),
        \end{split}
                \label{eq_duality_3}
        \end{equation}
        with equality if and only if $\PP K$ is a dilate of $\PP \G \PP K$. Next, let $L=\PP K$ in Lemma~\ref{lem_duality} and apply Minkowski's first inequality \eqref{eq:min_first} to obtain for every $K\in\conbod$ that
\begin{equation}
        \label{eq_duality_1}
        \left(\frac{n\hid+1}{\hid}\right)^{-n} \geq \vol_{n}(\G \PP K) \vol_{n}(K)^{n-1},
\end{equation}
    with equality if and only if $K$ is homothetic to $\G \PP K$. Raising both sides of \eqref{eq_duality_1} to the $\hid$th-power and combining with \eqref{eq_duality_3}, we then obtain \eqref{eq_duality_4} with equality if and only if $K$ is homothetic to $\G\PP K$ (from \eqref{eq_duality_1}) and $\PP K$ is a dilate of $\PP \G \PP K$ (from \eqref{eq_duality_3}).
	Notice the first condition implies the second, as the operator $\PP$ is translation invariant. Thus, it suffices to say there is equality if and only if $K$ is homothetic to $\G\PP K$. 
\end{proof}

In 1999, McMullen \cite{McM99} introduced the fiber combination of convex bodies. In 2016, Bianchi, Gardner and Gronchi \cite{BGG17} further generalized the concept of fiber combination and constructed a general framework of symmetrizations of convex bodies. We now introduce a symmetrization that is adapted to our setting, i.e. which uses the product structure of $\R^{nm}$. It is actually a particular fiber summation of certain bodies. The definition as a union also emphasizes that the symmetrization is a particular case of \cite[Equation 6]{BGG17}. As far as we know, this is the first application of this particular case. For lack of a better phrase, we will call this fiber symmetrization. It will be convenient to write $\bar x = (x_1,\dots,x_\hid)\in\Rnhi$ as $(x_i)_i.$

\begin{definition}\label{d:multidimSteinerSymmetrization}

Fix $\hid,n \in \N$. For $\xi\in\s$, consider the line $\langle \xi \rangle:=\{t\xi: t \in \R\} \subseteq \R^n$ and let $V=V(\xi)$ be its orthogonal complement.
Let $L \subseteq \Rnhi$ be a compact, convex set.
Consider the subspace $V^\hid = V \times \cdots \times V \subseteq \Rnhi$ and define similarly $\langle \xi \rangle^\hid$. We define the fiber symmetral of $L$ with respect to $\xi$ by
\[
\bar S_\xi L
= \bigcup_{\bar x \in V^\hid} \left( \bar x + \frac 12 D(L \cap (\bar x+ \langle \xi \rangle^\hid)) \right).
\]
We use the convention that the Minkowski sum of a point and the emptyset is the emptyset. Notice that the set $D(L \cap (\bar x+ \langle \xi \rangle^\hid))$ is an $\hid$-dimensional convex body inside the vector space $\langle \xi \rangle^\hid$. Furthermore, we have
\begin{equation}
\label{eq:sym_def}
\begin{split}
\bar S_\xi L=
\bigg\{\left(x_i+\frac 12(t_i-s_i) \xi\right)_i&\! \in \Rnhi \colon x_i \in V,
\\
&t_i,s_i \in \R, (x_i + t_i \xi)_i,(x_i + s_i \xi)_i \in L \bigg\}.
\end{split}
\end{equation}
\end{definition}
It is easy to see that if $L\in\Conbod$, i.e. if $L$ has non-empty interior, then so too is $\bar S_{\xi} L$. This can be verified directly or found in \cite{BGG17,McM99}. In the case when $\hid =1$, the set $L \cap (x+\langle \xi \rangle^\hid)$ is an interval, and half its difference body is the centered interval of the same length, parallel to $\langle \xi \rangle$. Consequently, the above definition reduces to the classical Steiner symmetrization of compact, convex sets. It is straightforward to verify that
\begin{equation}
    \label{eq_ellsym_invariance}
    \overline T(\bar S_{\xi} L) = \bar S_{T \xi} \overline T (L)
\end{equation}
for every rotation $T \in \mathcal{O}(n).$

A basic property of Steiner symmetrization is that it preserves volume. In \cite{UJ23}, Ulivelli studied the analogue of this property for the symmetrizations introduced by Bianchi, Gardner and Gronchi \cite{BGG17}. 
In the following proposition, we provide a self-contained proof of a particular case of \cite[Lemma 3.1]{UJ23}.
\begin{proposition}
\label{p:vol_increase}
Let $n,m \in \mathbb{N}$ and fix $\xi \in \s$. In the notation of Definition \ref{d:multidimSteinerSymmetrization}, we have
that, for $L\in\Conbod$, 
\begin{equation}
        \label{eq:steiner_vol_increase}
	    \vol_{n\hid}(\bar S_\xi L) \geq \vol_{n\hid}(L).
	\end{equation}

Equality holds if and only if the fibers $L \cap\left(\bar{x}+\langle\xi\rangle^{m}\right)$ with positive $m$-dimensional volume are symmetric for every $\bar{x} \in V^{m}$ (the center of symmetry may depend on $\bar{x}$ ).
\end{proposition}

\begin{proof}

By the Brunn-Minkowski inequality \eqref{eq:BM},
\[
\vol_{\hid}\left(\frac{1}{2} D(L \cap (\bar x+ \langle \xi \rangle^\hid ))\right) \geq \vol_{\hid}( L \cap (\bar x+ \langle \xi \rangle^\hid)).
\]
Applying Fubini's theorem, we obtain
\begin{align*}
    \vol_{n\hid}(\bar S_{\xi} L)
    &= \int_{V^\hid} \vol_{\hid}\left(\bar x+ \frac{1}{2} D(L \cap (\bar x+\langle \xi \rangle^\hid))\right) d\bar x \\
    &= \int_{V^\hid} \vol_{\hid}\left(\frac{1}{2} D(L \cap (\bar x+\langle \xi \rangle^\hid))\right) d\bar x \\
    &\geq \int_{V^\hid} \vol_{\hid}(L \cap (\bar x+\langle \xi \rangle^\hid)) d\bar x = \vol_{n\hid}(L),
\end{align*}
as required.
\end{proof}

\begin{remark}
As pointed out by the referee, it makes sense to apply fiber symmetrization from Definition~\ref{d:multidimSteinerSymmetrization} to any compact set $L\subset \Rnhi$. However, when $m=1$, one does not necessarily recover the Steiner symmetrization of $L\subset \R^n$. The inequality \eqref{eq:steiner_vol_increase} in Proposition~\ref{p:vol_increase} still holds, but the equality conditions become more complicated due to the more intricate nature of the equality conditions of the Brunn-Minkowski inequality (see e.g. \cite[Page 363]{gardner_book}): there is equality if and only if whenever $\vol_m(L\cap (\bar x +\langle \xi \rangle^m))>0$, the set $L\cap (\bar x +\langle \xi \rangle^m)$ is a convex set with a center of symmetry, from which subsets of $m$-dimensional volume zero may have been removed.
\end{remark}

In order to characterize the equality case of Theorem~\ref{t:pettyprojectioninequality}, we need the following lemma.
\begin{lemma}
    \label{l:permutationofvariables}
    Let $\Delta:\Rn  \to \Rnhi$ be the diagonal function given by $\Delta(x) = (x,\dots, x)$. Then, for $\xi\in\s$ and $K\in\conbod$:
    \[\Delta^{-1} (\bar S_\xi \PP K) \subseteq S_\xi \Pi^\circ K,\]
    where $\Delta^{-1}$ is the pre-image of $\Delta.$
    \end{lemma}
\begin{proof}
Notice that $\Delta^{-1}(\PP K) = \Pi^\circ K.$ Consider the left-inverse of $\Delta$,
\[A(\bar x) = \frac 1n \sum_{j=1}^\hid x_j, \quad \text{satisfying} \quad \Delta A(\bar x) = \frac 1{\hid!} \sum_{\sigma\in P_\hid} \sigma(\bar x),\]
where the second sum runs over $P_\hid,$ the collection of all permutations in $\hid$-coordinates. Since $\P  K$ is invariant under permutations in $P_\hid$, we have
\begin{align}
    h_{\P  K}(\bar\theta)
    &= \frac 1{\hid!} \sum_{\sigma\in P_\hid} h_{\P  K}( \sigma(\bar\theta) ) \geq h_{\P  K}\left( \frac 1{\hid!} \sum_{\sigma\in P_\hid} \sigma(\bar\theta) \right) \\
&= h_{\P K}( \Delta(A(\bar\theta)))= h_{\Pi K}(A(\bar\theta)),\\
\end{align}
which implies that
\[A(\PP K) \subseteq \Pi^\circ K.\]

To prove the lemma, take $x+r\xi \in \Delta^{-1}(\bar S_\xi \PP K )$ with $x \perp \xi$.
We have $x+r\xi = (x+\frac 12(t_i-s_i) \xi)$ for $i = 1, \ldots, \hid$ with
$(x+t_i\xi)_i, (x+s_i\xi)_i \in \PP K$
and $t_i-s_i = 2r$. Also notice that $A((x+t_i\xi)_i) = x + t \xi$ with $t = \frac 1n \sum t_i$, and $A((x+s_i\xi)_i) = x + s \xi$ with $s = \frac 1n \sum s_i$. Since $(x+t\xi),(x+s\xi) \in A(\PP K) \subseteq \Pi^\circ K$ and
\begin{align}
(x+r\xi)
&= x+ A(\Delta(r \xi)) \\
&= x + A\left((\frac 12(t_i-s_i)\xi)_i\right) \\
&= x + \frac 12 t \xi - \frac 12 s \xi
\end{align}
we obtain $x+r \xi = x+ \frac 12(t-s)\xi \in S_\xi \pp K$ and the lemma follows.
\end{proof}

The critical ingredient in the proof of Theorem~\ref{t:pettyprojectioninequality} is the next lemma, which is a $\hid$th-order variant of \cite[Lemma~14]{LYZ00}.

\begin{lemma}
	\label{l:multidimSteiner}
	Fix $\xi \in \s$.
	Given $K \in \conbod$ with $C^1$ smooth boundary, one has
	\[\bar S_\xi \PP K \subseteq \PP S_{\xi} K.\]

	\noindent If there is equality, then $S_{\xi} \pp K \supseteq \pp S_\xi K$.
\end{lemma}

\begin{proof}
By the invariance property \eqref{eq_ellsym_invariance}, we may assume that $\xi =e_n,$ where $e_n= (0,\ldots, 0,1)$. In view of \eqref{eq:sym_def} it is enough to show that if $((x_i, t_i))_i, ((x_i, s_i))_i \in \PP K$,
	then $((x_i, \frac 12(t_i - s_i) ))_i \in \PP S_{\xi} K$. Write
\[
K = \{(x,t) \in \R^{n-1} \times \R \colon \u z(x) \leq t \leq \o z(x), x \in P_{e_n^{\perp}}K\},
\]
where $\u z, \o z \colon \R^{n-1} \to \R$ are $C^1$ functions in the relative interior of $P_{e_n^{\perp}}K$, and we are identifying $e_n^\perp=\R^{n-1}$ and $x+te_n=(x,t)$. Then the (classical) Steiner symmetral in the direction $e_n$ has the form
\[S_{e_n} K = \{(x,t) \in \R^{n-1} \times \R \colon -z(x) \leq t \leq z(x), x \in P_{e_n^{\perp}}K \},
\]
where $z = \frac 12(\o z - \u z)$. Note that, since $K$ is $C^1$, so is $S_{e_n} K$.
Set $\theta_i = (x_i, t_i)$ and $\bar\theta=(\theta_1,\ldots,\theta_\hid)$. 
Notice for $x$ in the interior of $P_{e_n^\perp}K,$ the outer unit normal on $\partial K$ at the point $(x,\bar{z}(x))$ is given by
$$n_K(x,\bar{z}(x)) = \frac{(-\nabla \o z(x),1)}{|(-\nabla \o z(x),1)|}.$$
Similarly, 
 the outer unit normal on $\partial K$ at the point $(x,\u{z}(x))$ is given by
$$n_K(x,\u{z}(x)) = \frac{(\nabla \u z(x),-1)}{|(\nabla \u z(x),-1)|}.$$
Then, we have from a variable substitution that
	\begin{align*}
		h_{\P K}&(\bar\theta)
		= \int_{P_{e_n^{\perp}}K} \max_{1\leq i \leq \hid} \left\{\left\langle \theta_i, \frac{(-\nabla \o z(x),1)}{|(-\nabla \o z(x),1)|} \right\rangle_- \right\} |(-\nabla \o z(x),1)| dx \\
		&\quad\quad+ \int_{P_{e_n^{\perp}}K} \max_{1\leq i \leq \hid} \left\{\left\langle \theta_i, \frac{(\nabla \u z(x),-1)}{|(\nabla \u z(x),-1)|} \right\rangle_- \right\} |(\nabla \u z(x),-1)| dx \\
		&= \int_{P_{e_n^{\perp}}K}\! \left(\max_{1\leq i \leq \hid} \left\{\langle \theta_i, {(-\nabla \o z(x),1)} \rangle_- \right\}\! +\! \max_{1\leq i \leq \hid} \left\{\langle \theta_i, {(\nabla \u z(x),-1)} \rangle_- \right\} \right)dx \\
		&= \int_{P_{e_n^{\perp}}K}\! \left(\max_{1\leq i \leq \hid} \{0, \langle x_i, \nabla \o z(x) \rangle - t_i  \}\! + \!\max_{1\leq i \leq \hid} \{0, -\langle x_i, \nabla \u z(x) \rangle + t_i \} \right) dx, 
	\end{align*}
    where we used that $t_- = (-t)_+ = \max\{0,-t\}$.

	Now assume $h_{\P K}( (x_i, t_i)_i ), h_{\P K}((x_i, s_i)_i ) \leq 1$ and let $r_i = \frac{1}{2}(t_i - s_i)$. Continuing our computation, we have
	\begin{align*}
        h&_{\P S_{e_n}K}((x_i, r_i)_i)
        \\
		&= \int_{P_{e_n^{\perp}}K} \left(\max_{1\leq i \leq \hid} \{0, \langle x_i, \nabla z(x) \rangle - r_i  \} \!+\! \max_{1\leq i \leq \hid} \{0,-\langle x_i, \nabla (-z)(x) \rangle + r_i \}\right)  dx \\
		&=\frac{1}{2} \int_{P_{e_n^{\perp}}K} \bigg(\max_{1\leq i \leq \hid} \{0, \langle x_i, \nabla \o z(x) \rangle - \langle x_i, \nabla \u z(x) \rangle - t_i + s_i  \}  \\& \quad\quad+ \max_{1\leq i \leq \hid} \{0, \langle x_i, \nabla \o z (x) \rangle - \langle x_i, \nabla \u z (x) \rangle + t_i - s_i \} \bigg) dx \\
		&\leq\frac{1}{2} \int_{P_{e_n^{\perp}}K} \bigg(\max_{1\leq i \leq \hid} \{0, \langle x_i, \nabla \o z(x) \rangle - t_i \} + \max_{1\leq i \leq \hid} \{0, -\langle x_i, \nabla \u z(x) \rangle + s_i  \}  \\& \quad\quad+ \max_{1\leq i \leq \hid} \{0, -\langle x_i, \nabla \u z (x) \rangle + t_i \} + \max_{1\leq i \leq \hid}\{0, \langle x_i, \nabla \o z (x) \rangle - s_i \}  \bigg)dx
	\end{align*}
	where we used that $\max_i\{a_i+b_i\} \leq \max_i\{a_i\} + \max_i\{b_i\}$ for every $a_i, b_i \in \R$.
	We obtain
	\begin{align*}
        h&_{\P S_{e_n}K}((x_i, r_i)_i)
        \\
		&\leq \frac{1}{2} \int_{P_{e_n^{\perp}}K} \bigg(\max_{1\leq i \leq \hid} \{0, \langle x_i, \nabla \o z(x) \rangle-t_i \} + \max_{1\leq i \leq \hid} \{0, -\langle x_i, \nabla \u z (x) \rangle + t_i \} \\& + \max_{1\leq i \leq \hid}\{0, \langle x_i, \nabla \o z (x) \rangle - s_i \}   + \max_{1\leq i \leq \hid} \{0, -\langle x_i, \nabla \u z(x) \rangle  + s_i  \}\bigg)dx \\
		&= \frac{1}{2}( h_{\P  K}( (x_i, t_i)_i) + h_{\P  K} ((x_i, s_i)_i) ) \leq 1,
	\end{align*}
 which completes the proof of the first inclusion. 
 
  If $\bar S_\xi \PP K = \PP S_\xi K$, then by Lemma \ref{l:permutationofvariables} we have
	\[\Pi^\circ S_\xi K = \Delta^{-1} \PP S_\xi K = \Delta^{-1} \bar S_\xi \PP K \subseteq S_\xi \Pi^\circ K.\]
\end{proof}

We are now in a position to prove Theorem~\ref{t:pettyprojectioninequality}.

\begin{proof}[Proof of Theorem~\ref{t:pettyprojectioninequality}]
	First we prove that
	\[\vol_n(K)^{n\hid-\hid} \vol_{n\hid}(\PP K) \leq \vol_n(\B)^{n\hid-\hid} \vol_{n\hid}(\PP \B).\]
	In view of Proposition~\ref{p:continuityofmultdimproj}, we can assume that $K$ is a $C^1$ convex body.
	Combining \eqref{eq:steiner_vol_increase} and Lemma~\ref{l:multidimSteiner}, we observe that, for any given $\xi \in \s$, one has
	\begin{equation}
	    \label{eq:pettychainofinequalities}
     \begin{split}
	    \vol_n(K)^{n\hid-\hid}\vol_{n\hid}(\PP K) &\leq \vol_n(S_{\xi} K)^{n\hid-\hid}\vol_{n\hid}(\bar S_{\xi}\PP K) 
     \\
     &\leq \vol_n(S_{\xi} K)^{n\hid-\hid}\vol_{n\hid}(\PP S_{\xi}K).
     \end{split}
	\end{equation}
	Finally, following the notation of Section~\ref{sec:properties}, we choose a sequence of directions $\{\xi_j\}_j \subset \s$ such that $S_{j}K \to \kappa_n^{-1/n}\vol_n(K)^{1/n}\B$ in the Hausdorff metric.
	Again, by Proposition \ref{p:continuityofmultdimproj} we obtain the result.

	Now assume there is equality in \eqref{eq:pettychainofinequalities}. By Proposition~\ref{p:GammaC1}, we know that $K$ is a $C^1$ convex body.
	By Lemma \ref{l:multidimSteiner} and the equality of volumes, we have $\bar S_\xi \PP K = \PP S_\xi K$ for every $\xi \in \s$.
	By the equality case of Lemma \ref{l:multidimSteiner}, we have $\Pi^\circ S_\xi K \subseteq S_\xi \Pi^\circ K$ and, therefore,
    \[\Pi^\circ S_\xi K = S_\xi \Pi^\circ K\]
    for every $\xi \in \s$. This implies that $K$ is an extremal body for the classical Petty projection inequality and thus, an ellipsoid.
\end{proof}

From the mixed volume formula, we can also obtain an isoperimetric-type inequality, a $\hid$th-order analogue of Petty's isoperimetric inequality \eqref{eq:petty_theorem}.

\begin{reptheorem}{t:petty_classic}
    Let $K\in\conbod$ and $\hid\in\mathbb{N}.$ Then, one has the following inequality:
    \begin{align*}\vol_{n\hid}(\PP K)\vol_{n-1}(\partial K)^{n\hid}
    &\geq  \vol_{n\hid}(\PP B_2^n)\vol_{n-1}(\s)^{n\hid} 
    \\
    &\geq \kappa_{n\hid}\left(\frac{n\kappa_{n}}{w_{n\hid}(\P\B)}\right)^{n\hid}.
    \end{align*}

Equality in the first inequality holds if and only if $\Pi K$ is an Euclidean ball. If $\hid=1$, there is equality in the second inequality, while for $\hid \geq 2,$ the second inequality is strict.
\end{reptheorem}
\begin{proof}
    We begin with the first inequality.    
    Using the polar formula for volume \eqref{eq:polar_2}, one obtains 
    \begin{equation}\label{eq:polar_vol}
    \vol_{n\hid}(\PP K) = \frac 1{n\hid} \int_{\S} (n V_n(K[n-1], C_{-\bar\theta}))^{-n\hid} d\bar\theta.
    \end{equation}

     Note that for any $T \in \mathcal{O}(n)$, we have $\overline{T} \in \mathcal{O}(n\hid)$. Now, by a change of variables and integrating with respect to the Haar probability measure $\mu$ on $\mathcal{O}(n)$, we obtain
     
     \[\vol_{n\hid}(\PP K) = \frac 1{n\hid} \int_{\mathcal{O}(n)} \int_{\S} (n V_n(K[n-1], C_{-\overline T \bar\theta}))^{-n\hid} d\bar\theta d \mu(T).\]
     
     By Fubini's theorem, Jensen's inequality and the identity $C_{- \overline T \bar \theta} = T C_{-\bar \theta}$, we have
     \begin{align}
         \vol_{n\hid}&(\PP K) 
         \geq  \frac 1{n\hid} \int_{\S} \left( \int_{\mathcal{O}(n)} n V_n(K[n-1], T C_{-\bar\theta}) d \mu(T) \right)^{-n\hid} d\bar\theta \\
         &=  \frac 1{n\hid} \int_{\S} \left( \int_{\mathcal{O}(n)} \int_{\s} h_{C_{-\bar\theta}}(T^t u) d \sigma_K(u) d \mu(T) \right)^{-n\hid} d\bar\theta.
     \end{align}
     Using Fubini's theorem again, and the fact that $\mathcal{O}(n)$ acts uniformly on $\s$, 
     \begin{align}
         \vol_{n\hid}&(\PP K)
         \geq \frac 1{n\hid} \int_{\S} \left( \int_{\s} \int_{\mathcal{O}(n)} h_{C_{-\bar\theta}}(T^t u) d \mu(T) d \sigma_K(u)  \right)^{-n\hid} d\bar\theta \\
         &= \frac 1{n\hid} \int_{\S} \left( \frac 1{n\kappa_n} \int_{\s} \int_{\s} h_{C_{-\bar\theta}}(\nu) d \nu d \sigma_K(u)  \right)^{-n\hid} d\bar\theta \\
         &= \frac 1{n\hid} \int_{\S} \left( \vol_{n-1}(\partial K) \frac 1{\kappa_n} V_n(B_2^n[n-1],C_{-\bar\theta}) \right)^{-n\hid} d\bar\theta \\         &= \vol_{n-1}(\partial K)^{-n\hid} \frac{\kappa_n^{n\hid}}{n\hid} \int_{\S} V_n(B_2^n[n-1],C_{-\bar\theta})^{-n\hid}d\bar\theta,
     \end{align}
and we obtain
\begin{align*}\vol_{n\hid}(\PP K) &\vol_{n-1}(\partial K)^{n\hid}
\\
&\geq \vol_{n-1}(\s)^{n\hid} \frac 1 {n\hid} \int_{\S} (n V_n(B_2^n[n-1],C_{-\bar\theta}))^{-n\hid}d\bar\theta.
\end{align*}
     
    But by \eqref{eq:polar_vol}, this quantity is precisely $\vol_{n-1}(\s)^{n\hid} \vol_{n\hid}(\PP B_2^n)$. This yields the first inequality in the proposition.

    Equality holds in this inequality if and only if $V_n(K[n-1], C_{-\overline T \bar\theta})$ is independent of $T$ for almost every $\bar\theta$. By continuity (see the proof of Proposition \ref{p:continuityofmultdimproj}), in this case it actually holds that $V_n(K[n-1], C_{-\overline T \bar\theta})$ is independent of $T$ for any $\bar\theta$. In particular this is true if we take $\bar\theta =(\theta, o, o, \ldots, o)$ for some $\theta \in \s$, in which case
    $V_n(K[n - 1], C_{-\overline T \bar\theta}) = h_{\Pi K}(T\theta)$. Thus $\Pi K$ is a ball.    

    For the second inequality, we return to \eqref{eq:polar_vol} and now apply Jensen's inequality to the integral over $\S$, giving
    $$\vol_{n\hid}(\PP K) \geq  (\kappa_{n\hid}n\hid)^{n\hid}\kappa_{n\hid}\left(\int_{\S}\int_{\s} h_{C_{-\bar\theta}}(u)d\sigma_K(u)d\bar\theta\right)^{-n\hid}.$$

    \noindent If $\hid=1,$ then $\int_{\s}h_{[o,-\theta_1]}(u)d\sigma_K(u)$ is constant for all $\theta_1,$ if and only if $K$ is a multiple of $B_2^n.$ For $\hid\geq 2,$ the corresponding integral is never constant for all $\bar\theta.$ Next, it is not hard to see from the rotational invariance of the spherical Lebesgue measure that $\int_{\S}h_{C_{-\bar\theta}}(u)d\bar\theta = \text{constant},$
    i.e., this integral is independent of $u$. Thus, from Fubini's theorem, one obtains
    $$\vol_{n\hid}(\PP K)\vol_{n-1}(\partial K)^{n\hid} \geq  (\kappa_{n\hid}n\hid)^{n\hid}\kappa_{n\hid}\left(\int_{\S} h_{C_{-\bar\theta}}(u)d\bar\theta\right)^{-n\hid}.$$
    We now replace $\int_{\S} h_{C_{-\bar\theta}}(u)d\bar\theta$ with something more geometric in meaning. Since this is a constant, we can integrate it over $\s$ and obtain
    $$\int_{\S} h_{C_{-\bar\theta}}(u)d\bar\theta=\frac{1}{n\kappa_n}\int_{\s}\int_{\S} h_{C_{-\bar\theta}}(u)d\bar\theta du.$$
    However, from a use of Fubini's theorem, this becomes
    $$\int_{\S} h_{C_{-\bar\theta}}(u)d\bar\theta =\frac{1}{n\kappa_n}\int_{\S}h_{\P \B}(\bar\theta)d\bar\theta=\frac{\hid \kappa_{n\hid}}{\kappa_n}w_{n\hid}(\P \B).$$
    
    \noindent Inserting this computation yields the result.
\end{proof}

\section{The Busemann-Petty Centroid Inequality and Random Processes Without Independence}
\label{sec:cent}

We start by obtaining the $\hid$th-order Busemann-Petty centroid inequality, Theorem~\ref{t:BPCH}, as a direct corollary of the $\hid$th-order Petty's projection inequality. We list it here again for the convenience of the reader.

\begin{reptheorem}{t:BPCH}[The mth-order Busemann-Petty centroid inequality] 
For $L\in\sta$, where $n,\hid\in\mathbb{N},$ one has,
    % $$\vol_n(\G  L) \geq \left(\frac{\hid}{(n\hid +1)\kappa_n}\right)^{n}\kappa_n\left(\frac{\vol_{n\hid}(L)}{\vol_{n\hid}(\PP \B)}\right)^\frac{1}{\hid},$$
    \[\frac{\vol_n(\G  L)}{\vol_{n\hid}(L)^{1/\hid}} \geq \frac{\vol_n(\G \PP\B)}{\vol_{n\hid}(\PP\B)^{1/\hid}}, \]
    with equality if and only if $L=\PP E$ for any ellipsoid $E\in\conbod.$
\end{reptheorem}
\begin{proof}
    Applying Theorem~\ref{t:pettyprojectioninequality} to the body $\G L,$ one has the bound
    $$\vol_{n\hid}(\PP \G  L)^{-1} \geq \vol_n(\G  L)^{n\hid-\hid}\kappa_n^{\hid-n\hid}\vol_{n\hid}(\PP \B)^{-1},$$
    with equality if and only if $\G L$ is an ellipsoid.
    Combining this bound with \eqref{eq_duality_2} and the fact that Lemma~\ref{l:class} shows $$\vol_{n}(\G \PP \B)=\kappa_n\left(\frac{\hid}{(n\hid+1)\kappa_n}\right)^n$$ yields the inequality. The equality conditions are inherited from Theorem~\ref{t:pettyprojectioninequality} and Lemma~\ref{l:class}.
\end{proof}

\begin{remark}
Given a compact set $K\subset \R^n$, its moment body $M K$ is defined as 
$$h_{MK}(\theta)=\int_L|\langle x,\theta \rangle|dx \longleftrightarrow MK=\vol_n(K)\g K.$$
Similarly, we can define the $\hid$th-order moment body of a compact set $L\subset \Rnhi$ as the $C^1$ convex body (from Proposition~\ref{p:GammaC1}) $M^\hid L$ in $\Rn$ whose support function is given by
$$h_{M^\hid L}(\theta)= \int_L \max_{1\leq i \leq \hid} \langle x_i, \theta \rangle_{-} d\bar{x} \longleftrightarrow M^\hid L=\vol_{n\hid}(L)\G L.$$
From Proposition~\ref{p:cen_tran}, one has
$$ M^\hid  \overline{T} L = |\det T|^{\hid}TM^\hid L$$
for $T\in GL_n(\R)$.
The Busemann-Petty inequality for the mth-order moment body follows from Theorem~\ref{t:BPCH} via homogeneity: for $L\in\sta,$
\[\frac{\vol_n(M^\hid  L)}{\vol_{n\hid}(L)^{n+\frac{1}{\hid}}} \geq \frac{\vol_n(M^\hid \PP\B)}{\vol_{n\hid}(\PP\B)^{n+\frac{1}{\hid}}},\]
with equality if and only if $L=\PP E$ for any ellipsoid $E\in\conbod.$
\end{remark}

We now work towards proving Theorem~\ref{t:randomsimplex1}. We first need the following proposition.
\begin{proposition} For $L\subset \Rnhi$ a compact set with positive volume, $K\in\conbod,$ and $\bar x\in\Rnhi$ one has
\label{p:randomsimplex1}
	\[V_n(K [n-1], \G L) \!=\!\frac{1}{\vol_{n\hid}(L)} \int_L V_n(K [n-1], C_{-\bar x}) d\bar x \!=\! \E_L (V_n(K[n-1], C_{-\bar X}) ).\]
\end{proposition}
\begin{proof}
	The result follows from the definition of mixed volumes, the definition of $\G L,$ and Fubini's theorem. Indeed:
 \begin{align*}
V_n(K[n-1],\G L)&=\frac{1}{n}\int_{\s}h_{\G L}(\theta)d\sigma_K(\theta)
\\
&=\frac{1}{n}\int_{\s}\frac{1}{\vol_{n\hid}(L)}\int_{L}\max_{1 \leq i \leq \hid}\langle x_i,\theta\rangle_-d\bar xd\sigma_K(\theta)
\\
&=\frac{1}{\vol_{n\hid}(L)}\int_{L}\frac{1}{n}\int_{\s}h_{C_{-\bar x}}d\sigma_K(\theta)d\bar x
\\
&=\frac{1}{\vol_{n\hid}(L)} \int_L V_n(K [n-1], C_{-\bar x}) d\bar x.
 \end{align*}
\end{proof}
We now prove Theorem~\ref{t:randomsimplex1}, which we first restate.
\begin{reptheorem}{t:randomsimplex1}
	The functional
	\[(K,L) \in \conbod \times \sta \mapsto \vol_{n\hid}(L)^{-\frac{1}{n\hid}}\vol_n(K)^{- \frac{n-1}n} \E_L (V_n(K[n-1], C_{\bar X}) )\]
	is uniquely minimized when $K$ is an ellipsoid and $L = \lambda \PP K$ for some $\lambda >0$.
\end{reptheorem}
\begin{proof}
	By Proposition \ref{p:randomsimplex1}, Minkowski's first inequality for mixed volumes \eqref{eq:min_first} and Theorem \ref{t:BPCH},
	\begin{align}
		 \vol_{n\hid}(L)^{-1-\frac{1}{n\hid}}&\vol_n(K)^{- \frac{n-1}n}\int_L V_n(K[n-1], C_{\bar x}) d\bar x
   \\
   &=\vol_{n\hid}(L)^{-1-\frac{1}{n\hid}}\vol_n(K)^{- \frac{n-1}n}\int_L V_n(-K[n-1], C_{-\bar x}) d\bar x
   \\
		&=  \vol_{n\hid}(L)^{-\frac{1}{n\hid}}\vol_n(K)^{- \frac{n-1}n}V_n(-K[n-1], \G L) \\
		&\geq \vol_{n\hid}(L)^{-\frac{1}{n\hid}}\vol_n(\G L)^\frac{1}{n}
        \\
        &\geq \vol_{n\hid}(\PP\B)^{-\frac{1}{n\hid}}\vol_n(\G \PP\B)^{\frac{1}{n}}.
	\end{align}
	From the equality conditions of Minkowski's first inequality, there is equality if and only of $K$ and $\G L$ are homothetic. From Theorem~\ref{t:BPCH}, there is equality if and only if $L = \overline T (\PP \B)$.
	By Lemma~\ref{l:class}, this happens if and only if $K$ is an ellipsoid and $L$ is a dilate of $\PP K$.
\end{proof}
We next show we can generalize these results. Minkowski \cite{Min1} showed that volume behaves as a polynomial: given $\lambda_1,\dots,\lambda_n >0 $ and $K_1,\dots, $ $K_n \in \conbod,$ one has that
$$
\vol_n(\lambda_1 K_1+\lambda_2 K_2+\cdots+\lambda_n K_n)=\!\sum_{i_1, i_2, \ldots, i_n=1}^n V_n\left(K_{i_1}, \ldots, K_{i_n}\right) \lambda_{i_1} \lambda_{i_2} \ldots \lambda_{i_n} .
$$
Then $V_n\left(K_{i_1}, \ldots, K_{i_n}\right)$ is called the mixed volume of $K_{i_1}, \ldots, K_{i_n}$.

\noindent We can then repeat the above procedure and obtain a Groemer-like result for these larger order mixed volumes. We first need the following proposition.
\begin{proposition}\label{p:randomsimplex2} For $K_1,\ldots,K_{n-1}\in\conbod$ and a compact set $L\subset \Rnhi$ with positive volume,  one has
	\[V_n(K_1, \ldots, K_{n-1}, \G L) = \frac{1}{\vol_{n\hid}(L)} \int_L V_n(K_1, \ldots, K_{n-1}, C_{-\bar x}) d\bar x,\]
	\[\vol_{n}(\G L) = \frac{1}{\vol_{n\hid}(L)^n}\int_L \cdots \int_L V_n(C_{-\bar{x}_1}, \ldots, C_{-\bar{x}_n}) d\bar x_1 \ldots d\bar x_n.\]
\end{proposition}
\begin{proof}
	
 It is well-known that the mixed volumes have an integral representation similar to that of \eqref{eq:mixed_0}, except instead of $\sigma_K,$ one integrates with respect to the \textit{mixed surface-area measures} \cite[Chapter 5]{Sh1}: consider a convex body $K$ in $\R^n$ and a collection of $(n-1)$ convex bodies $\mathfrak{K}=(K_2,\dots,K_n),K_i\in\conbod$. Then,
 \begin{equation}
 \label{eq:mixed_most}
     V_n(K_1,\mathfrak{K})=\frac{1}{n}\int_{\s}h_{K_1}(u)d\sigma_{\mathfrak{K}}(u),
 \end{equation}
 where $\sigma_{\mathfrak{K}}$ is the mixed surface-area measure of the collection $\mathfrak{K}$.
 Thus, the proof of the first identity is similar to that of Proposition~\ref{p:randomsimplex1}. As for the second identity, use \[\vol_n(\G L) = V_n(\G L, \cdots, \G L)\] and apply the first identity $n$-times.
\end{proof}
Finally, given $L \in \sta$, we select independently $n$ points $\bar X_i \in \Rnhi$ uniformly distributed inside $L$.
We denote the expected mixed volume of the closed convex hulls by
\[\E_{L^\hid}( V_n(C_{\bar X_1}, \cdots, C_{\bar X_n} ) ),\]
and we obtain the following extension of the Busemann random simplex inequality.
\begin{corollary}
	\label{t:randomsimplex2}
	The functional
	\[L \in \sta \mapsto \vol_{n\hid}(L)^{-\frac{1}{\hid}}\E_{L^\hid} ( V_n(C_{\bar X_1}, \cdots, C_{\bar X_n} ) )\]
	is minimized exactly when $L = \PP E$, where $E$ is an ellipsoid.
\end{corollary}
\begin{proof}
	From Proposition \ref{p:randomsimplex2}, one obtains $$\vol_n(\G (-L))= \E_{L^\hid} ( V_n(C_{\bar X_1}, \cdots, C_{\bar X_n} ) ).$$ From the translation invariance of volume, the result therefore follows from Theorem \ref{t:BPCH} applied to $-L.$
\end{proof}

 \section{The mth-order affine Sobolev inequality} \label{sec:Sobolev Inequality}
 
Throughout this section, we say that a function $f:\Rn\to\R$ is non-identically zero if $\vol_n(\left\{x\in\Rn:f(x) \neq 0\right\}) >0$. Recall that $f\in BV(\R^n)$, the space of functions of bounded variation, if, by definition, there exists a vector valued measure $Df$ such that, if $\eta_f:\Rn\to \Rn$ is the Radon-Nykodim derivative of $Df$ with respect to $|Df|$ (the variation measure of $f$), then
$$\int_{\Rn} f(x) \text{ div}( \psi(x) ) dx = - \int_{\Rn} \langle \eta_f (x), \psi(x) \rangle d|Df|(x)$$
holds for every compactly supported, smooth vector field $\psi:\Rn \to \Rn$. In the case when $f$ is differentiable, or in the Sobolev space $W^{1,1}(\Rn)$, one has, $d|Df|(x)=|\nabla f(x)|dx$ and $\eta_f=\nabla f / |\nabla f|$ when $\nabla f \neq o,$ and $o$ otherwise (see \cite{EvansGar}).

It was shown in \cite{TW12} that given a non-identically zero function $f \in BV(\Rn)$, there exists a unique Borel measure $\mu_f$ on $\s$, not concentrated on any great hemisphere, such that
\begin{equation}\int_{\s}g(u)d{\mu_f}(u)=\int_{\Rn}g(-\eta_f(x))d|Df|(x)
\label{eq:LYZog}
\end{equation}
for every $1$-homogeneous function $g$ on $\Rn\setminus\{o\}$. But also, the center of mass of $\mu_{f}$ is the origin. Indeed, with $\{e_{i}\}$ the canonical basis of $\Rn$,
\begin{align*}
\int_{\s} u d\mu_{f}(u)&=\sum_{i=1}^n\left(\int_{\s} \langle u,e_{i} \rangle d\mu_{f}(u)\right)e_{i}
\\
&=-\sum_{i=1}^n\left(\int_{\Rn} \langle \eta_f (x),e_{i} \rangle d|Df|(x)\right)e_{i}
\\
&=\sum_{i=1}^n\left(\int_{\Rn}  f(x) \; \mathrm{div} (e_{i}) dx\right)e_{i} =\sum_{i=1}^n 0e_{i}= o.
\end{align*}
Thus, we apply Minkowski's existence theorem to $\mu_{f}$ and obtain a unique convex body with center of mass at the origin, the asymmetric LYZ body of $f$ denoted $\langle f \rangle$, such that $\sigma_{-\langle f \rangle}=\mu_f$. That is, $\langle f \rangle$ satisfies the following change of variables formula for every $1$-homogeneous function $g$ on $\Rn\setminus\{o\}$:
\begin{equation}\int_{\s}g(u)d\sigma_{\langle f\rangle}(u)=\int_{\Rn}g(\eta_f(x))d|Df|(x).
\label{eq:LYZ}
\end{equation}

We remark that the epithet ``LYZ Body of $f\in BV(\Rn)$" in \cite{TW12} is reserved for the origin-symmetric convex body whose surface area measure is the \textit{even} part of $\sigma_{\langle f\rangle}$, in which case \eqref{eq:LYZ} would only hold for even, 1-homogeneous functions $g$ (this distinction is made implicitly in the proof of \cite[Theorem 3.1]{TW12} but not in the statement of that theorem). In fact, this origin-symmetric convex body is the Blaschke body of $\langle f\rangle$ (see \eqref{eq:BL_body}). The not necessarily symmetric body is important for our setting. Indeed, by setting $g(u)=h_{C_{-\bar\theta}}(u)$ in \eqref{eq:LYZ}, which is in general not an even function in $u$, we obtain the \textit{LYZ projection body of order $\hid$ of $f\in BV(\Rn)$}, $\Pi^\hid \langle f\rangle$, defined via
\begin{equation}
h_{\LYZ{f}}(\bar \theta) = \int_{\R^n} \max_{1\leq i \leq \hid}\langle \eta_f(x), \theta_i \rangle_-d|Df|(x).
\end{equation}

We need the following lemma. In the case of Sobolev functions, it was done in by Federer and Fleming \cite{FF60} and Maz'ja \cite{VGM60} when $K=\B$. Gromov later generalized this to when $K$ is any symmetric convex body \cite{MS86}. Cordero-Erausquin, Nazaret and Villani then extended Gromov's result, with equality conditions, in the case when $f\in BV(\Rn)$ \cite[Theorem 3]{CENV04}. The version we present here, which holds for all convex bodies $K$ containing the origin in their interior, follows by either modifying the proof of \cite[Theorem 3]{CENV04} slightly, or, alternatively, from the result by Alvino, Ferone, Trombetti and Lions for Sobolev functions \cite[Corollary 3.2]{AFTL97}, an approximation argument involving $W^{1,1}(\Rn)$, and repeating the argument for equality conditions done in \cite{CENV04}.
\begin{lemma}[Anisotropic Sobolev inequality]
\label{l:asob}
Let $K$ be a convex body in $\R^n$ containing the origin in its interior. Consider a compactly supported, non-identically zero function $f\in BV(\Rn)$.
Then, 
$$\|f\|_{\frac{n}{n-1}}\vol_n(K)^{\frac{1}{n}} \leq \frac{1}{n} \int_{\Rn} h_K(\eta_f(z))d|Df|(z),$$
with equality if and only if there exists a convex body $L$ homothetic to $K$ and a constant $A>0$ such that $f=A\chi_L$.
\end{lemma}

\begin{theorem}\label{t:GeneralAffineSobolev} Fix $\hid,n \in \N$. Consider a compactly supported, non-identically zero function $f\in BV(\Rn)$. Then, one has the geometric inequality
\[
\|f\|_{\frac{n}{n-1}} \vol_{n\hid}(\LYZP{f})^\frac{1}{n\hid} \leq \vol_{n\hid}(\PP{\B})^\frac{1}{n\hid}\kappa_n^{\frac{n-1}{n}}.
\]
Equivalently, by setting $
d_{n,\hid} := n\kappa_n\left(nm\vol_{n\hid}(\PP{\B})\right)^\frac{1}{n\hid}$, this is
\begin{equation*}
\left(\int_{\S} \left(\int_{\R^n} \max_{1\leq i \leq \hid}\langle \eta_f(z), \theta_i \rangle_- d|Df|(z) \right)^{-n\hid} d\bar\theta\right)^{-\frac{1}{n\hid}} d_{n,\hid} \geq  n\kappa_n^\frac{1}{n}\|f\|_{\frac{n}{n-1}}.
\end{equation*}
In either case, there is equality if and only if there exists $A>0$, and an ellipsoid $E\in\conbod$ such that
$
f(x)=
A\chi_{E}(x).
$
\end{theorem}

\begin{proof}
From Theorem~\ref{t:pettyprojectioninequality} applied to $\langle f\rangle$, we obtain
$$\vol_n(\langle f\rangle)^{\frac{n-1}{n}}\vol_{n\hid}(\PP \langle f\rangle)^\frac{1}{n\hid} \leq \vol_{n}(\B)^\frac{n-1}{n}\vol_{n\hid}(\PP \B)^\frac{1}{n\hid}.$$
From an application of Lemma~\ref{l:asob} with $K=\langle f\rangle$, \eqref{eq:LYZ} and \eqref{eq:mixed_0}, we obtain
 $$\|f\|_{\frac{n}{n-1}}\leq \vol_n(\langle f\rangle)^{\frac{n-1}{n}}.$$
 This establishes the geometric inequality. Using the polar formula \eqref{eq:polar_2} for the volume of $\LYZP{f}$ yields the other inequality.
 
 As for the equality conditions, there is equality in the use of Theorem~\ref{t:pettyprojectioninequality} if and only if $\LYZP{f} = \PP{E}$ for some centered ellipsoid $E \in \conbod$. First, we show that $\langle f \rangle $ is a translate of $E$. If $m\geq 2$, this is immediate from Proposition~\ref{prop:proj_uniqueness}. If $m=1$, we take polarity and deduce that $ \Pi \langle f\rangle=\p E$. It follows by Aleksandrov's projection theorem \cite[Theorem 3.3.6, page 142]{gardner_book} combined with \cite[Theorem 4.1.3, page 143]{gardner_book} that this means $\tilde {\langle f\rangle}$, the origin-symmetric Blaschke-body of $\langle f \rangle$ defined via \eqref{eq:BL_body}, equals $E$. Since $\Pi \langle f\rangle=\Pi \tilde{\langle f\rangle}$, the equation \eqref{eq:BL_volume} yields
 \begin{align*}
     \vol_n(\B)^{n-1}\vol_n(\pp\B)&=\vol_n(\langle f\rangle)^{n-1}\vol_n(\pp \langle f\rangle) 
     \\
     &= \vol_n(\langle f\rangle)^{n-1}\vol_n(\pp \tilde{\langle f\rangle})
     \\
     &\leq  \vol_n(\tilde{\langle f\rangle})^{n-1}\vol_n(\pp \tilde{\langle f\rangle})
     \\
     &\leq \vol_n(\B)^{n-1}\vol_n(\pp\B),
 \end{align*}
 and, therefore, there is equality throughout. By the equality conditions of \eqref{eq:BL_volume}, $\tilde{\langle f\rangle}=E$ and $\langle f\rangle$ is a translate of $E$.
 
 As for the use of Lemma~\ref{l:asob}, there is equality if and only if $f$ is a multiple of $\chi_{L},$ where $L$ is homothetic to $\langle f \rangle$. Thus, $L$ is homothetic to $E$ and the claim follows.
\end{proof}

A Borel set $D\subset \Rn$ is said to be \textit{a set of finite perimeter} if $\chi_D \in BV(\Rn)$. Examples include convex bodies and compact sets with $C^1$ boundary. Given such a set $D$, we define its mth-order polar projection body via the Minkowski functional
$$\|\bar\theta\|_{\PP D} = \int_{\partial^\star D} \max_{1 \leq i \leq \hid} \langle \theta_i, n_D(y) \rangle_{-}dy, \quad \bar\theta\in\S,$$
where $\partial^\star D$ is the reduced boundary of $D$ and $n_D:=-\eta_{\chi_D}$ is its measure-theoretic outer-unit normal, see e.g. \cite[Definition 2.2]{TW12}. A standard mollifying argument by H\"ormander \cite{HL} (using the fact that continuously differentiable, compactly supported functions are dense in $BV(\Rn)$) immediately yields the following extension of Theorem~\ref{t:pettyprojectioninequality} to sets of finite perimeter by setting $f=\chi_D$ in Theorem~\ref{t:GeneralAffineSobolev}; see \cite[Theorem 7.2]{TW12} for the $\hid=1$ case.
\begin{corollary}
Let $\hid \in \N$ be fixed. Then, for every set of finite perimeter $D\subset\Rn,$ one has 
\[
\vol_n(D)^{n\hid-\hid}\vol_{n\hid}(\PP D) \leq \vol_{n}(\B)^{n\hid-\hid}\vol_{n\hid}(\PP \B),
\]
with equality if and only if $D$ is an ellipsoid up to a set of measure zero. If $D$ is unbounded, then the left-hand side is taken to be $0$.
\end{corollary}

Observe that applying Theorem~\ref{t:petty_classic} with $K=\langle f\rangle$ for $f\in BV(\Rn)$ yields 
\begin{equation}
\vol_{n\hid}(\LYZP f)\vol_{n-1}(\partial \langle f \rangle)^{n\hid}
    \geq  \vol_{n\hid}(\PP B_2^n)\vol_{n-1}(\s)^{n\hid}
    \label{eq:Sobolev_almost_classic}
    \end{equation}
    with equality if and only if $\langle f\rangle$ is a centered Euclidean ball. Combining this with Theorem~\ref{t:GeneralAffineSobolev}, one obtains an isoperimetric inequality for the LYZ body:
    \begin{equation}
    \vol_{n-1}(\partial \langle f \rangle) \geq n\|f\|_{\frac{n}{n-1}}\kappa_n^{\frac{1}{n}},
    \label{eq:LYZ_iso}
    \end{equation}
    with equality if and only if $f$ is a multiple of a characteristic function of a centered Euclidean ball.
    Using the integral formula for mixed volumes \eqref{eq:mixed_0} and the change of variables formula satisfied by functions of bounded variation \eqref{eq:LYZ}, we obtain
    \begin{equation}
    \label{eq:LYZ_ISO_int}
    \begin{split}
\vol_{n-1}(\partial \langle f \rangle)&=nV_n(\langle f\rangle[n-1],\B)=\int_{\s}d\sigma_{\langle f\rangle}(u)
\\
&=\int_{\Rn}d|Df|(x)=TV(f),
\end{split}
    \end{equation}
where $TV(f)$ is the total variation of $f\in BV(\Rn)$. If  $f \in W^{1,1}(\Rn)$, then $TV(f)=\|\nabla f\|_1$. In other words, Theorem~\ref{t:GeneralAffineSobolev} in conjunction with Theorem~\ref{t:petty_classic} implies the classical Sobolev inequality for every choice of $\hid\in\N$. By setting $f=\chi_D$ in \eqref{eq:LYZ_iso}, we also obtain the classical isoperimetric inequality for sets of finite perimeter.

\backmatter

\begin{appendices}

\section{The mth-order polar projection body versus product sets}\label{secA1}

The first way one can construct a higher-order Petty projection inequality is to ``embed" the Petty projection inequality from $\Rn$ into $\Rnhi$ via the direct product $K\mapsto (K\times\cdots\times K)= (K)^\hid.$ Then, this``embedded'' higher-order Petty's projection inequality is
\begin{align*}
\vol_{n\hid}((K)^\hid)^{n-1}\vol_{n\hid}((\pp K)^\hid) \leq (\vol_{n\hid}((\B)^\hid)^{n-1}\vol_{n\hid}((\pp \B)^\hid)).
\end{align*}
We show that the $\hid$th-order Petty projection inequality Theorem~\ref{t:pettyprojectioninequality} is sharper than this embedded version.
\begin{theorem} For $K\in\conbod$ and $\hid\in\N,$ one has \[\PP K \subseteq (\pp K)^\hid \quad \text{and} \quad \vol_{n\hid}(\PP K) \leq \vol_n(\pp K)^{\hid},\]
with equality if and only if $\hid=1.$
\label{t:compare}
\end{theorem}

\begin{proof} For each $\bar{\theta} \in \mathbb{S}^{n\hid-1}$, given by $\bar{\theta} = (\theta_1,\dots,\theta_{\hid})$, notice that 
\begin{align*}
h_{\P K}(\bar{\theta}) &= \int_{\s} \max_{1 \leq i \leq \hid} \langle u,\theta_{i} \rangle_- d\sigma_K(u)\geq \max_{1 \leq i \leq \hid} \left[\int_{\s} \langle u,\theta_{i} \rangle_- d\sigma_K(u) \right]\\
&= \max_{1 \leq i \leq \hid}h_{\p K}(\theta_i) = h_{\text{conv}(\p K,\cdots,\p K)}(\bar{\theta}). 
\end{align*}
Applying duality, this above computation implies that 
\[
\PP K \subseteq \underbrace{\pp K \times \cdots \times \pp  K}_{\hid-\text{times}}=:(\pp K)^\hid,
\]
where each $\pp K$ belongs to an independent copy of $\Rn$. 

We now show for $\hid \geq 2$ that $  \vol_{n\hid}\left((\pp K)^\hid\setminus \PP K\right) > 0.$ From the continuity of radial functions, it suffices to find a single $\bar x \in \Rnhi\setminus\{o\}$ such that $\rho_{\PP K}(\bar x) < \rho_{(\pp K)^\hid}(\bar x).$ From duality, it suffices to show that $h_{\P K}(\bar x) > h_{\text{conv}(\Pi K,\cdots,\Pi K)}(\bar x).$ We will actually show there are many such $\bar x,$ to illustrate that $(\pp K)^\hid$ and $\PP K$ differ ``by a lot''.  First, define the following sets:
\begin{align*}
\Sigma&=\left\{\bar x\in\Rnhi: x_1,x_2\in\Rn\setminus\{o\}, x_1\neq x_2, x_3=\dots=x_\hid=o\right\},
\\
\Sigma_{a}(\bar x)&=\left\{u\in\s: \langle x_1,u \rangle_- > \langle x_2,u \rangle_- \right\}, \; \text{for a fixed }\bar x \in \Sigma,
\\
\Sigma_{b}(\bar x)&=\left\{u\in\s: \langle x_1,u \rangle_- \leq \langle x_2,u \rangle_- \right\}, \;\text{for a fixed } \bar x \in \Sigma.
\end{align*}
Notice for a fixed $\bar x\in\Sigma,$ $\Sigma_{a}(\bar x)\cap \Sigma_{b}(\bar x)=\emptyset$ and $\Sigma_{a}(\bar x)\cup \Sigma_{b}(\bar x)=\s.$ Furthermore, it is possible to pick $\bar x \in \Sigma$ such that $\sigma_K(\Sigma_a),\sigma_K(\Sigma_b)>0.$ 
We will let $\Sigma^\prime \subseteq \Sigma$ denote the collection of $\bar x$ with this property. One such example is when $x_1=-x_2,$ for any $x_2\in\R^n\setminus\{o\},$ since $\sigma_K$ is not concentrated on any great hemisphere. 

We now compute: for every $\bar x \in \Sigma^\prime,$
\begin{align*}
h_{\P K}(\bar x) &= \int_{\s} \max_{1 \leq i \leq \hid} \langle u,x_{i} \rangle_- d\sigma_K(u)= \int_{\s} \max_{i=1,2} \langle u,x_{i} \rangle_- d\sigma_K(u)
\\
&= \int_{\Sigma_{a}(\bar x)} \langle u,x_1 \rangle_- d\sigma_K(u) +  \int_{\Sigma_{b}(\bar x)} \langle u,x_2 \rangle_- d\sigma_K(u).
\end{align*}
Applying the definition of $\Sigma_{a}(\bar x)$ yields
\begin{align*}
h_{\P K}(\bar x) &> \int_{\Sigma_{a}(\bar x)} \langle u,x_2 \rangle_- d\sigma_K(u) +  \int_{\Sigma_{b}(\bar x)} \langle u,x_2 \rangle_- d\sigma_K(u) 
\\&= \int_{\s}\langle u,x_2 \rangle_- d\sigma_K(u).
\end{align*}
Similarly, applying the definition of $\Sigma_{b}(\bar x)$ yields
\begin{align*}
h_{\P K}(\bar x) &> \int_{\Sigma_{a}(\bar x)} \langle u,x_1 \rangle_- d\sigma_K(u) +  \int_{\Sigma_{b}(\bar x)} \langle u,x_1 \rangle_- d\sigma_K(u)
\\
&= \int_{\s}\langle u,x_1 \rangle_- d\sigma_K(u).
\end{align*}
In particular, this yields, for every $\bar x \in \Sigma^\prime,$
\begin{align*}
h_{\P K}(\bar x) & > \max_{i=1,2} \left[\int_{\s} \langle u,x_{i} \rangle_- d\sigma_K(u) \right]\\
&= \max_{1 \leq i \leq \hid}h_{\p K}(x_i)=h_{\text{conv}(\p K,\cdots,\p K)}(\bar{x}). 
\end{align*}
We remark that there is nothing special about $x_1$ and $x_2$; one could pick any $i,j\in\{1,\dots,\hid\}, i\neq j,$ and define $\Sigma_{i,j}^\prime$ analogously to $\Sigma_{1,2}^\prime:=\Sigma^\prime.$ Then, $$h_{\P K}(\bar x) > h_{\text{conv}(\p K,\cdots,\p K)}(\bar{x}) \quad  \text{for every} \quad \bar x \in \underset{\substack{1 \leq i, j \leq \hid \\ i\neq j}}{\bigcup}\Sigma_{i,j}^\prime.$$
\end{proof}
\noindent Consequently, we have that
$$\vol_n(K)^{n\hid-\hid}\vol_{n\hid}(\PP K)\leq (\vol_n(K)^{n-1}\vol_n(\pp K))^\hid,$$
with equality if and only if $\hid=1.$
Combining this estimate with Theorem~\ref{t:pettyprojectioninequality} we obtain the following.
\begin{corollary}
\label{cor:sharper_ineq}
    For $K\in\conbod$ and $\hid\in\N,$ one has that
    \begin{align*}
\vol_n(K)^{n\hid-\hid}\vol_{n\hid}(\PP K) &\leq \vol_{n}(\B)^{n\hid-\hid}\vol_{n\hid}(\PP \B) 
\\
&\leq (\vol_n(\B)^{n-1}\vol_n(\Pi^{\circ} \B))^{\hid},
\end{align*}
with equality in the first inequality if and only if $K$ is an ellipsoid, and equality in the second inequality if and only if $\hid=1.$
\end{corollary}

\end{appendices}

\section*{Declarations}

{\bf Funding:} The first named author was supported by Grants RYC2021-031572-I and PID2022-136320NB-I00, funded by the Ministry of Science and
 Innovation and by the E.U. Next Generation EU/Recovery, Transformation and Resilience Plan. The second named author was supported in part by the U.S. NSF Grant DMS-2000304 and by the Chateaubriand Scholarship offered by the French embassy in the United States. The fifth named author was supported in part by NSERC. Additionally, part of this work was completed while the second, third and fourth named authors were in residence at the Institute for Computational and Experimental Research in Mathematics in Providence, RI, during the Harmonic Analysis and Convexity program; this residency was supported by the National Science Foundation under Grant DMS-1929284. The fourth named author was additionally supported by the Simons Foundation under Grant No. 815891 for this stay.

\vspace{5mm}
\noindent  {\bf Conflict of interests :} There are no conflict of interests.

\noindent {\bf Data Availability :} There is no data associated with this project.

\bmhead{Acknowledgments}

We extend heartfelt thanks to Matthieu Fradelizi, Richard Gardner, Erwin Lutwak, Jacopo Ulivelli and Artem Zvavitch for the helpful comments concerning this work. Finally, we are deeply grateful to the anonymous referees for their diligent efforts and insightful comments, which significantly improved this work.

% \newpage

\bibliography{references}

\end{document}